\documentclass[10pt]{amsart}

\usepackage{geometry}
\usepackage{amsmath,amssymb, xy,bm,amsthm,graphicx,hyperref,mathrsfs}
\usepackage[capitalize,nameinlink,noabbrev,nosort]{cleveref}

\newtheorem{theorem}{Theorem}[section]
\newtheorem{lemma}[theorem]{Lemma}
\newtheorem{proposition}[theorem]{Proposition}

\newtheorem{remark}[theorem]{Remark}
\newtheorem{conjecture}[theorem]{Conjecture}
\newtheorem{property}{Property S}

\newtheorem{corollary}[theorem]{Corollary}

\newcommand{\R}{\mathbb R}
\newcommand{\Z}{\mathbb Z}
\newcommand{\bla}{\bm{\lambda}}
\newcommand{\Grn}{G(r,1,n)}

\DeclareMathOperator{\dih}{Dih}
\newcommand{\ds}{\displaystyle}
\newcommand{\Irr}{\mathrm{Irr}}

\newcommand{\Conj}{\mathrm{Conj}}

\newcommand{\B}{\mathcal{B}}

\begin{document}

\title[Character degree sum versus the character table sum]
{
How large is the character degree sum compared to the character table sum for a finite group?
}

\date{\today}
\author{Arvind Ayyer}
\address{Arvind Ayyer, Department of Mathematics, Indian Institute of Science, Bangalore  560012, India.}
\email{arvind@iisc.ac.in}

\author{Hiranya Kishore Dey}
\address{Hiranya Kishore Dey, Department of Mathematics, Indian Institute of Science, Bangalore  560012, India.}
\email{hiranyadey@iisc.ac.in}

\author{Digjoy Paul}
\address{Digjoy Paul, Department of Mathematics, Indian Institute of Science, Bangalore  560012, India.}
\email{digjoypaul@iisc.ac.in}

\begin{abstract}
In 1961, Solomon gave upper and lower bounds for
the sum of all the entries in the character table of a finite group 
in terms of elementary properties of the group. In a different direction, we consider the ratio of the character table sum to the sum of the entries in the first column, also known as the character degree sum, in this work. 
First, we propose that this ratio is at most two for many natural groups. Secondly, we extend a conjecture of Fields to postulate that this ratio is at least one with equality if and only if the group is abelian.
We establish the validity of this property and conjecture for all finite irreducible Coxeter groups. In addition, we prove the conjecture for generalized symmetric groups.
The main tool we use is that the sum of a column in the character table of an irreducible Coxeter group (resp. generalized symmetric group) is given by the number of square roots (resp. absolute square roots) of the corresponding conjugacy class representative. 

As a byproduct of our results, we show that the asymptotics of  character table sums is the same as the number of involutions in symmetric, hyperoctahedral and demihyperoctahedral groups.
We also derive explicit generating functions for the character table sums for these latter groups as infinite products of continued fractions.
In the same spirit, we prove similar generating function formulas for the number of square roots and absolute square roots in $n$ for the generalized symmetric groups $\Grn$.
\end{abstract}



\maketitle

\section{Introduction}

For any finite group, it is natural to consider the sum of the entries of the character table. 
Solomon~\cite{Solomon} proved that this is always a nonnegative integer by proving something stronger, namely that all row sums are nonnegative integers.
He did so by showing that the sum of the row indexed by an irreducible representation is the multiplicity of that representation in the conjugation representation of that group. 
He then deduced that the sum of the entries in the character table of a finite group is at most the cardinality of the group and at least the cardinality of a maximal abelian normal subgroup of the group.
It was shown that the multiplicities in the regular and conjugacy representations for the symmetric group are close in a certain sense~\cite{Frumkin-86,Adin-Frumkin,roichman-1997}.

In this article, we take a different approach to estimating the sum of the entries of the character table by considering column sums instead. 
It is well known that the column sums are always integers, though not necessarily nonnegative~\cite[Problem~(5.13)]{isaacs-2006}. 
Alternating groups are an example of a family of groups for which some column sums in the character table are negative. 
We do not know of a classification of groups for which all column sums are nonnegative integers. However, for groups whose irreducible characters are real, the column sums are given by the number of square roots of conjugacy class representatives by a classical result by Frobenius and Schur~\cite{isaacs-2006}.
This class of groups is called totally orthogonal (see \cref{rem:tot orth}). Finite Coxeter groups are well-known examples of such groups.  
When $r \geq 3$, generalized symmetric groups $\Grn$ are not totally orthogonal, though their column sums are nonnegative integers. This is because these column sums are
given by the number of so-called absolute square roots~\cite{adin-postnikov-roichman-2010}.

For a finite group $G$, let $\Irr(G)$ denote the set of irreducible complex representations upto isomorphism, and $\Conj(G)$ denote the set of conjugacy classes of $G$. Let $\chi_V$ denote the character associated with the representation $V$ and $\chi_V(C)$ be the value of the character $\chi_V$ evaluated at an element of the conjugacy class $C$. The character table of $G$ is a square matrix encoding character values, whose rows are indexed by  $\Irr(G)$ and columns are indexed by $\Conj(G)$. 
Let $s(G)$ denote the sum of all entries in the character table of $G$. Let $R_{V}$ (resp. $\Gamma_{C}$) be the sum of all the entries in the row (resp. column) of the character table indexed by $V$ (resp. $C$). By convention, the first row is indexed by the trivial character and the first column is indexed by the singleton class containing the identity element $e$ of $G$. Thus, the first column records the dimensions of corresponding irreducible representations and $\Gamma_{e}(G)$ is their sum.
It is a standard fact (for instance, see~\cite[Lemma 2.15]{isaacs-2006}) that $\Gamma_e \geq |\Gamma_C|$ for any group and any conjugacy class $C$.
We note that $\Gamma_e$ has been used to obtain structural properties of groups. Starting with the work of Magaard--Tong-Viet~\cite{magaard-tongviet-2011}, there is a large literature of works in that direction.

From extensive computations, we observe the following upper bound for the sum of the entries of the character table for many but not all groups.

\begin{property}
\label{property:S} 
For a finite group $G$, we have $s(G) \leq 2 \Gamma_e(G)$.
\end{property}

We know that \cref{property:S} will not hold in general, but it seems to hold for a large class of natural groups. We have looked at data for all groups of order up to $200$ using the \texttt{SmallGroups} Library in \texttt{GAP}.
The first counterexamples occur at order $64$. See~\cref{sec:prop-fails} for the list of groups indexed according to \texttt{GAP} where \cref{property:S} fails. 

Fields~\cite{Fields75} provided two-sided bounds for the total sum
in terms of the number (and orders of corresponding centralizers) of conjugacy classes, cardinality of the group and the first column sum.
The first inequality in the following conjecture is due to him; see \cite[Remark 3 after Theorem 2]{Fields75}. We strengthen his conjecture to propose a characterization of when equality is attained.

\begin{conjecture}
\label{conj:lower bound}
For all finite groups $G$, we have $s(G) \geq \Gamma_e(G)$. In addition, $s(G) = \Gamma_e(G)$ if and only if $G$ is abelian.
\end{conjecture}

The equality holds for abelian groups; see \cref{prop:abelian}. Note that 
the inequality in the conjecture follows immediately for all finite groups $G$ 
for which $\Gamma_C \geq 0$ for all conjugacy classes $C$ in $G$. 
Totally orthogonal groups and certain complex reflection groups (see \cref{rem: pos col sum}) are examples of such groups. We prove \cref{conj:lower bound} for these groups in \cref{prop: col pos} and \cref{thm: conj Grn} respectively using explicit counting formulas for $\Gamma_{C}$.
Using Solomon's result, we also deduce that \cref{conj:lower bound} holds for any finite group of nilpotency two; see~\cref{cor:nilpotent}. In addition, we have verified that \cref{conj:lower bound} holds for all groups of order up to 200. See \cref{sec:ratios} for the list of ratios of $s(G)/\Gamma_e(G)$ in increasing order. 

Our main result is for finite irreducible Coxeter groups.
Recall that Coxeter groups are abstract generalizations of reflection groups~\cite{humphreys-1990}.
Every finite Coxeter group is a direct product of finitely many irreducible Coxeter groups. Finite irreducible Coxeter groups consist of four one-parameter familes: symmetric groups $S_n$ (type $A$), hyperoctahedral groups $B_n$ (type $B$), demihyperoctahedral groups $D_n$ (type $D$) and dihedral groups $\dih(n)$ (type $I_2(n)$). There are six more finite irreducible Coxeter groups of type $E_{6},E_{7},E_{8},F_{4},H_{3},H_{4}$, which are called exceptional groups. It is well known that every Weyl group can be realized as a Coxeter group and they are $S_n,B_n,D_n, E_{6},E_{7},E_{8},F_{4}$ and $G_2 \cong \dih(6)$. 

\begin{theorem}
\label{thm:cox}
\cref{property:S} holds for all finite irreducible Coxeter groups.
\end{theorem}

The irreducibility condition in the statement is necessary; see \cref{prop: conj false}. The proof of \cref{thm:cox} follows by a case analysis.
In \cref{sec:finite} we will sketch a proof of \cref{thm:cox} for dihedral groups $\dih(n)$.
We will prove \cref{property:S} for the symmetric, hyperoctahedral and demihyperoctahedral groups in the later sections. By explicit computations, we have verified the result for exceptional irreducible finite Coxeter groups; see \cref{sec:exceptional-groups}. 
None of the counterexamples are simple groups and it is tempting to believe that \cref{property:S} holds for all finite simple groups. We have not yet done a systematic study in that direction, but we certainly believe the following.

\begin{conjecture}
\cref{property:S} holds for all alternating groups.
\end{conjecture}

The plan of the article is as follows. We first present results for general finite groups in \cref{sec:finite}.
While proving  \cref{thm:cox} for the infinite one-parameter families of irreducible Coxeter groups, we also consider the sequence of character table sums for these families. In \cref{sec:sym}, we focus on $S_n$.
We compute the generating function of the sequence of these character table sums for $S_n$ in \cref{sec:Sn gf}. In \cref{sec:Sn asymp}, we prove \cref{thm:cox} for $S_n$ and show that the sequence grows as fast as the character degree sum in $S_n$.
We prove similar results for  $B_n$ in \cref{sec:Bn}
and for $D_n$ in \cref{sec:Dn}. We extend the generating function results to $\Grn$ in two ways in \cref{sec:wreath}. We first give explicit product formulas for the generating functions for the sum of the number of square roots for conjugacy class representatives in \cref{sec:Grn-sqrt}. We then do the same for the number of absolute square roots (which are also column sums) in \cref{sec:Grn-abs-sq}.

\section{Character theory of finite groups}
\label{sec:finite}

Any finite group $G$ has two natural representations afforded by the action of $G$ on itself. The first is the \emph{regular representation} and is given by left multiplication $g\cdot x=gx$. The second is the \emph{conjugation representation} and is given by $g\cdot x= gxg^{-1}$. The regular representation of a finite group is distinguished by the property that each irreducible representation appears within it with a multiplicity equal to its dimension.
Solomon~\cite{Solomon} obtained the analogous decomposition for the conjugation representation.

Recall that a group $G$ is said to be \emph{nilpotent of class two} if the derived subgroup $[G,G]$ is contained in the center of $G$.

\begin{theorem}[\cite{Solomon}]
	\label{thm:row-sums-non-neg} 
	The multiplicity of an irreducible representation $V$ of a finite group $G$ in the conjugation representation is the row sum $R_V$. 
	Furthermore,  if $h$ is the order of a maximal
	abelian normal subgroup of $G$  then $h \leq s(G) \leq |G|$. 
	The equality $s(G)=h$  holds if and only if $G$ is abelian and the equality $s(G)=|G|$ holds if and only if $G$ is nilpotent of class two.
\end{theorem}

In particular, the row sums of the character table of a finite group are nonnegative integers. We note in passing that classifying groups for which all the row sums are positive (equivalently, every irreducible representation appears in its conjugation representation) is an active area of research. This class includes symmetric and alternating group as shown by Frumkin~\cite{Frumkin-86} (also see~\cite{Sundaram-2018} for a modern treatment) and most of the finite simple groups of Lie type (see~\cite{Heide-Saxl-Tiep-Zalesski-2013}).

\begin{corollary}
\label{cor:nilpotent}
\cref{conj:lower bound} holds for any finite group of nilpotency class two.
\end{corollary}

\begin{proof}
Suppose $G$ is nilpotent of class two. Then $s(G)=|G|$, by \cref{thm:row-sums-non-neg}. Let $d_1,\ldots,d_r$ be the character degrees of $G$. Then, we have $\Gamma_e(G)=\sum^r_{i=1} d_i \leq \sum^r_{i=1} d_i^2 =|G|$. Since $d_i$'s are positive integers, the equality $s(G)=\Gamma_e(G)$ implies that $d_i=1$ for all $i$. Therefore, $G$ must be abelian.
\end{proof}

Given $U, V, W \in \Irr(G)$, let $g(U, V, W)$ be the multiplicity of $W$ in the tensor product representation $U \otimes V$. In the case of the symmetric group $S_n$, where the irreducible representations $U, V, W$ are indexed by integer partitions $\lambda,\mu,\nu$ of $n$, the multiplicity  $g(U, V, W)$ is the so-called \emph{Kronecker coefficient} $g(\lambda, \mu, \nu)$. Thus, we call $g(U, V, W)$ as the \emph{generalized Kronecker coefficient}.

\begin{remark}
We have the following character identity due to Frame~\cite{Frame-1947} (later proved by  Roth~\cite[Theorem 1.2]{Roth-1971} using Solomon's result):
	\[
	\chi_{\text{conj}}=\sum_{U\in \Irr(G)} \chi_U \chi_{U^*},
	\]
	where $\text{conj}$ is the conjugation representation and $U^*$ is the dual representation of $U$. Equating multiplicities of an irreducible character  $\chi_V$ in both sides of the above character identity, we obtain    
	\[
	R_{V}=\sum_{U \in \Irr(G)} g(U,U^*,V).
	\]
Therefore, the sum of all entries in the character table of $G$ is $\sum_{U, V \in \Irr(G)} g(U, U^*, V)$.
\end{remark}

Using Galois theory, one can prove that the column sums of the character table are integers, and in general, these sums can be negative as well, see~\cite[Problem~(5.13)]{isaacs-2006}. However, the following classical result gives a formula for the column sum for certain specific groups in terms of the square root function; see for instance~\cite{isaacs-2006}.

Given a finite group $G$, the \emph{Frobenius--Schur index} of an irreducible representation $V$ is defined by
\[
\sigma(V):= \frac{1}{|G|} \sum_{g \in G} \chi_V(g^2).
\]

\begin{theorem}[{\cite[Theorem 4.5]{isaacs-2006}}]
\label{thm:FS}
For a finite group $G$, the following holds:
\begin{enumerate}
	\item For each $g \in G$, we have 
	\[
	|\{x \in G \mid x^2=g\}|=\sum_{V \in \Irr(G)}\sigma(V)\chi_V(g).
	\]
	\item For $V \in \Irr(G)$, we have $\sigma(V)=1,0$ or $-1$  if $\chi_V$ is real-valued and $V$ is realizable over $\R$ (in which case, we say $V$ is \emph{real}), $\chi_V$ is not real-valued (in which case, we say $V$ is \emph{complex}), or $\chi_V$ is real-valued but $V$ is not realizable over $\R$ (in which case, we say $V$ is \emph{quaternionic}), respectively. 
	\end{enumerate}
\end{theorem}

\begin{remark}
\label{rem:tot orth}
A group $G$ is called \emph{totally orthogonal} if all its irreducible representations	 are real. All finite Coxeter groups are included in this class~\cite[Section 8.10]{humphreys-1990}.
Applying \cref{thm:FS} for such a group $G$, we get
\[
|\{x \in G \mid x^2=g\}|=\sum_{V \in \Irr(G)}\chi_V(g).
\] 
Thus, column sums of the character table of totally orthogonal groups are given by the number of square roots of conjugacy class representatives.
\end{remark}

\begin{proposition}
	\label{prop: col pos}
	Suppose $G$ is a finite group such that all column sums in its character table are nonnegative integers. Then
	\begin{enumerate}
		\item $s(G)\geq \Gamma_e(G)$.
		\item In addition, if $G$ is totally orthogonal, \cref{conj:lower bound} holds for $G$.
	\end{enumerate}  
\end{proposition}

\begin{proof}
The first part follows immediately as column sums are nonnegative. Let $s(G)=\Gamma_e(G)$. Then $\Gamma_C(G)=0$ for any non-trivial conjugacy class $C$. If $G$ is totally orthogonal, by \cref{rem:tot orth}, we have $g^2=e$ for all $g\in G$. 
This implies $G$ is abelian. 
\end{proof}

\begin{proposition}
	\label{prop: sum for product}
	For finite groups $G_1, G_2$, we have $s(G_1 \times G_2)=s(G_1)s(G_2)$ and $\Gamma_e(G_1 \times G_2)=\Gamma_e(G_1)\Gamma_e(G_2)$.
\end{proposition}

\begin{proof}
	The proof follows from the fact that $\Irr(G_1 \times G_2)=\{V_1 \otimes V_2 \mid V_1 \in \Irr(G_1), V_2 \in \Irr(G_2)\}$ and $\chi_{V_1 \otimes V_2}=\chi_{V_1}\chi_{V_2}$.
\end{proof}

\begin{corollary}
If \cref{conj:lower bound} holds for $G_1$ and  $G_2$, then it holds for $G_1 \times G_2$. 
\end{corollary}

\begin{proposition}
	\label{prop: conj false}
	For an integer $m>1$, there exists a group $G$ for which $s(G) > m \, \Gamma_e(G)$.
\end{proposition}

\begin{proof}	
	Let $n$ be a positive integer such that $2^n >m$. Let $H$ be any group satisfying $s(H)=2\Gamma_e(H)$; a list of such groups up to order $200$ is given in \cref{sec:prop-equal}. Applying \cref{prop: sum for product} we have
	$s(H^n)=s(H)^n=2^n \Gamma_e(H^n)>m \Gamma_e(H^n)$. This completes the proof.
\end{proof}

By \cref{prop: conj false}, \cref{property:S} does not also hold for all finite Coxeter groups.

%

For an abelian group $H$, the orthogonality of rows in the character table of $H$ leads to the vanishing of row sums of all representations except the trivial one. So, we have the following result.

\begin{proposition}
\label{prop:abelian}
For an abelian group $H$, we have $s(H) = \Gamma_e(H)$. 
\end{proposition}

\begin{corollary}
\label{cor: direct product}
Suppose $G$ satisfies \cref{property:S} and $H$ is an abelian group. Then $G \times H$ also satisfies \cref{property:S}.
\end{corollary}

\begin{proof}
As stated above, for an abelian group $H$, we know $s(H)=|H|=\Gamma_e(H)$. 
Thus $s(G \times H)=s(G)s(H)\leq 2\Gamma_e(G)\Gamma_e(H)=2\Gamma_e(G\times H)$.
\end{proof}

\begin{proposition}
\label{prop:all-irred-char-dim-1or2}
Let $G$ be a finite group such that all the irreducible characters of $G$ have degree at most $2$. Then $G$ satisfies \cref{property:S}. 	
\end{proposition}

\begin{proof}
Suppose $G$ has $r$ irreducible characters and 
$d_1, d_2, \dots, d_r$ 
are the degrees of the irreducible characters of $G$. 
We know that $\sum_{i=1}^{r} d_i^2=|G|$. 
Moreover, for each $1 \leq i \leq r$, we have $1 \leq d_i \leq 2$. 
Therefore, for each $1 \leq i \leq r$, we have $ d_i^2 \leq 2d_i$. 
Thus, 
we have 
\[\sum_{i=1}^{r} d_i^2 \leq 2 \sum_{i=1}^{r} d_i. \]
By \cref{thm:row-sums-non-neg}, 
we have 
\[s(G) \leq |G| \leq 2 \sum_{i=1}^{r} d_i = 2 \Gamma_{e}(G).\] 
This completes the proof. 
\end{proof}

A characterization of finite groups $G$ such that all the irreducible characters of $G$ have degree at most $2$ is given in \cite{Amitsur}.

\begin{lemma}[{\cite [p. 25]{Serre}}]
\label{lemma:Serre}
Let $A$ be an abelian subgroup of a finite group $G$. Then each irreducible character of $G$ has degree $\leq |G|/|A|$.  
\end{lemma}

Let $\Z_n \equiv \Z/n \Z$ be the cyclic group of order $n$.
For a finite abelian group $H$, the \emph{generalized dihedral group} $\dih(H)$ is defined as the semidirect product of $H$ and $\Z_2$, where $\Z_2$ acts on $H$ by inverting elements. Thus $\dih(H):= \Z_2\ltimes H$.
The usual dihedral group is thus $\dih(n) \equiv \dih(\Z_n)$ of order $2n$. 
Recall that the standard presentation of $\dih(n)$ is $\langle r,s \mid r^n=s^2=1, srs=r^{-1}\rangle$. The involutions of this group are precisely $r^isr^{-i}$ for $i=0,1,\ldots n-1$ for any $n$ and $r^{n/2}$ in addition when $n$ is even. Since $\dih(n)$ is a finite Coxeter group, by \cref{rem:tot orth} the first column sum of its character table is $n+1$ (resp. $n+2$) when $n$ is odd (resp. even).
But this sum is more than the sum of the remaining columns, which is at most $n-1$, proving \cref{property:S}.
We note that the character table sum of the dihedral group is given by~\cite[Sequence A085624]{oeis}
\[
\begin{cases}
	\ds	\frac{3n+1}{2} & \text{$n$ odd}, \\[0.2cm]
	\ds	\frac{3n+2}{2} & \text{$n \equiv 2 \pmod 4$,} \\[0.2cm]
	\ds	\frac{3n+4}{2} & \text{$n \equiv 0 \pmod 4$}.
\end{cases}
\]
Thus, the ratio $s(\dih(n))/\Gamma_e(\dih(n))$ tends to $3/2$ as $n$ tends to infinity.

Let $H$ be a finite abelian group of order $2n$ and assume $t$ is a unique element in $H$ of order $2$. Let $\Z_4 = \langle y \rangle$, and let $y$ act on $H$ by inverting elements. 
Then, the \emph{generalized quarternion group} $Q(H)$ is defined as the quotient of the semi-direct product $\Z_4\ltimes H$ by the two element central subgroup $\{1,y^2t\}$. Thus, $Q(H):=\Z_4\ltimes H/{\{1,y^2t\}}$ has order $4n$ with
the presentation,
\[
	Q(H)=\left< H,y \mid y^4=1,\textup{$yhy^{-1}=h^{-1}$ for $h\in H$},\ y^2t=1\right>.
\]
Let $H=\Z_4$, and take $\pm i$ to be the generators of $H$ and $y=j$, $t=-1$. Then $Q(H)$ is the usual quaternion group $Q_8=\{\pm 1,\pm i,\pm j,\pm k\}$.

As an application of \cref{lemma:Serre} and  \cref{prop:all-irred-char-dim-1or2}, we have the following:

\begin{corollary}
Let $H$ be a finite abelian group. Then the generalized dihedral group $\dih(H)$ satisfies \cref{property:S}. In addition, if $H$ has even order, the generalized quaternion group $Q(H)$ satisfies \cref{property:S}. 
\end{corollary}

\section{Symmetric groups}
\label{sec:sym}
Let $S_n$ be the symmetric group on $n$ letters. The set of irreducible representations and the conjugacy classes of $S_n$ are indexed by set of integer partitions $\lambda=(\lambda_1\geq \lambda_2\geq \cdots\geq \lambda_n)$ of $n$, denoted $\lambda \vdash n$. Write a partition in frequency notation as 
\begin{equation}
\label{freq not}
\lambda=\langle 1^{m_1}, \dots, n^{m_n} \rangle,
\end{equation}
 here  $m_i \equiv m_i(\lambda)$ denote the number of parts of length $i$ in $\lambda$. We are interested in $s_n$, the sum of the entries of the character table of $S_n$. The first twelve terms of the sequence $(s_n)$ for $n \geq 1$ are given by~\cite[Sequence A082733]{oeis}
\[
1, 2, 5, 13, 31, 89, 259, 842, 2810, 10020, 37266, 145373.
\]
No formula is given for this sequence as of this writing. Fix a representative of cycle type $\lambda$ by
\[
w_{\lambda}=(1,2, \ldots, \lambda_1)(\lambda_1+1, \lambda_1+2, \ldots, \lambda_1+\lambda_2)\cdots.
\]
Let $\Gamma_{\lambda}$ be the sum of entries of the column indexed by $\lambda$. 
By \cref{thm:FS} we have
\begin{equation}
	\label{eq:FS}
	\Gamma_{\lambda}=\sum_{\mu \vdash n} \chi_{\mu}(w_{\lambda})=|\{x\in S_n \mid x^2=w_{\lambda}\}|,
\end{equation}
where $\chi_{\mu}$ is the character of the irreducible representation $V_{\mu}$ of $S_n$ associated to $\mu$. In particular, we have $i_n=\sum_{\mu \vdash n} \dim V_{\mu}$, where $i_n$ is the number of \emph{involutions} 
 in $S_n$. 

\begin{remark}
\begin{enumerate}

\item Notice that the function $T: S_n \longrightarrow \Z$ defined by $T(w)=|\{x\in S_n \mid x^2=w\}|$ is a character of $S_n$ and $T=\sum_{\mu \vdash n} \chi_{\mu}$ by \eqref{eq:FS}. One can lift this character identity to the ring of symmetric polynomials via the Frobenius characteristic map. More details can be found in \cite[Exercise 7.69]{ec2}. 
	
\item Adin--Postnikov--Roichman~\cite{Adin-Postnikov-Roichman-2008}construct a representation $V_n$, based on the involutions in $S_n$, whose character $\chi_{V_n}=T$, thus proving $V_n \cong \underset{\lambda \vdash n}{\oplus}V_{\lambda}$ is a Gelfand model.
With this notation, \cref{property:S} for $S_n$ is equivalent to 
	 \[
	 \chi_{V_n}(e) \geq \sum_{\substack{\mu \vdash n \\ \mu \neq (n)} } \chi_{V_n}(w_{\mu}).
	 \]

\end{enumerate}
 \end{remark}

Recall that the \emph{double factorial} of an integer $n$ is given by $n!! = n (n - 2) \cdots $ ending at either 2 or 1 depending on whether $n$ is even or odd respectively. Define
\begin{equation}
	\label{num odd cycle}
	o_r(n) = \sum_{k = 0}^{\lfloor n/2\rfloor} \binom{n}{2k} (2k - 1)!! \, r^k.
\end{equation}
The following result is present in \cite{Adin-Postnikov-Roichman-2008}, but we give an independent proof.	

\begin{proposition}[{\cite[Corollary 3.2]{Adin-Postnikov-Roichman-2008}}]
	\label{prop:colsum-lambda}
	Suppose $\lambda \vdash n$ as written in \eqref{freq not}. Then the column sum $\Gamma_{\lambda}$ is 0 unless $m_{2i}$ is even for all $i \in \{1, \dots, \lfloor n/2 \rfloor\}$. If that is the case,
	\[
	\Gamma_{\lambda}= 
	\prod_{i = 1}^{\lfloor n/2 \rfloor} (m_{2i} - 1)!! \, (2i)^{m_{2i}/2}
	\prod_{j = 0}^{\lfloor (n-1)/2 \rfloor} o_{2j+1}(m_{2j+1}).
	\]
\end{proposition}

\begin{proof}
	Suppose we write $\mu \vdash n$ in frequency notation  given by $\mu = \langle 1^{p_1}, \dots, n^{p_n} \rangle$,
	and let $\pi \in S_n$ be a permutation with cycle type $\mu$. Observe that when we square $\pi$, even cycles will split into two cycles of half their length and odd cycles will remain odd cycles of the same length.
	Therefore $\pi^2$ has cycle type
	\[
	\langle 1^{p_1 + 2p_2}, 2^{2p_4}, 3^{p_3 + 2p_6}, 
	4^{2 p_8}, \dots \rangle.
	\]
	Therefore, a permutation with cycle type $\lambda$ written as in \eqref{freq not} has a square root if and only if 	$m_{2i}$ is even for all $i$. 
	Now, let us suppose this is the case.
	First note that cycles of different lengths do not interfere with each other in the sense that the number of square roots of cycles  of different length can be computed independently.
	
	We will look at odd and even cycles separately. First
	consider $m_{2i}$ cycles of length $2i$, where $m_{2i}$ is even. Then the square root of a permutation with this cycle type will have $m_{2i}/2$ cycles of length $4i$, as argued above. This will involve pairing these $(2i)$-cycles. This can be done in $(m_{2i}-1)!!$ ways. Further, for every pairing, the cycles can be joined in $2i$ ways. To see this, consider a single pair,
	\[
	(s_1, \dots, s_{2i}) \quad \text{and} \quad
	(t_1, \dots, t_{2i}).
	\]
	Then a square root merges these cycles with $s_j$'s and $t_j$'s being present alternately in that order. However, there is a freedom to choose the starting location of $t_1$ relative to $s_i$, and there are $2i$ choices. Since this choice is present for each pair, we get another factor of $(2i)^{m_{2i}/2}$.
	
	Now, consider $m \equiv m_{2j+1}$ cycles of length $2j+1$ for $j \geq 0$. We can form square roots of a permutation with this cycle type in two ways. First, by taking the cycle individually, and second, by merging two such cycles as in the even case. 
	Moreover, we have the freedom to choose as many pairs as we like. We can merge $k$ pairs by first choosing $2k$ cycles, combining them in $(2k-1)!!$ ways and then choosing the location for merging in $2j+1$ ways for each pair. Therefore, the total number of ways of forming square roots is
	\[
	\sum_{k = 0}^{m/2} \binom{m}{2k} (2k-1)!! \, (2j+1)^k.
	\]
	Comparing this with \eqref{num odd cycle}, we see that this is exactly $o_{2j+1}(m)$.
	Note that the case of $j = 0$ corresponds to involutions.
	This completes the proof.
\end{proof}

Bessenrodt--Olsson~\cite{bessenrodt-olsson-2004} proved the following 
conjecture of V. Jovovic~\cite[Sequence~A085642]{oeis}.

\begin{corollary}[{\cite{bessenrodt-olsson-2004}}]
\label{cor:bes-ols}
The number of columns in the character table of $S_n$ that have sum zero is equal to the number of partitions of $n$ with at least one part congruent to $2 \pmod 4$.
\end{corollary}

We sketch a proof of this conjecture because some of the ideas used here will be replicated later.
Let $\Lambda$ be the set of all integer partitions and $\Lambda_n$ be the partitions of $n$. Define the subsets 
\begin{align}
\label{lambda1}
\Lambda_n^1 := & \{ \lambda \vdash n \mid m_{2i}(\lambda) \text{ is even } \forall i\} \quad \text{ and } \\
\label{lambda2}
\Lambda_n^2 := & \{ \lambda \vdash n \mid m_{4i+2}(\lambda)=0 \, \forall i\},
\end{align}
of $\Lambda_n$.

\begin{proof}
The map $f: \Lambda_n^1 \to \Lambda_n^2$ given by
\begin{equation}
\label{bij12}
\langle 1^{m_1}, 2^{2m_2}, 3^{m_3}, 4^{2m_{4}}, \dots \rangle 
\overset{f}{\mapsto} 
\langle 1^{m_1}, 4^{m_2}, 3^{m_3}, 8^{m_{4}}, \dots \rangle
\end{equation}
gives a bijection between the two sets. By \cref{prop:colsum-lambda}, 
$\Lambda_n^1$ indexes the columns of the character table of $S_n$ with nonzero sum  proving the result.
\end{proof}

Bessenrodt--Olsson~\cite{bessenrodt-olsson-2004} found the following result, which we will also generalize to other groups.

\begin{proposition}[\cite{bessenrodt-olsson-2004}]
	\label{prop:bes-ols}
	The generating function for the number of columns of the character table of $S_n$ with nonzero sum is
	\[
	\prod_{j=1}^{\infty}\frac{1}{(1-q^{2j-1})(1-q^{4j})}.
	\]
\end{proposition}

\subsection{Generating function for the total sum of the character table} 
\label{sec:Sn gf}

We will first give some background on generating functions which are expressed as continued fractions. 
One of the original motivations for this theory is the fact that moments of many families of classical orthogonal polynomials have explicit combinatorial interpretations. For example, the $n$'th moment of the $n$'th Laguerre polynomial is $n!$. Therefore, it is natural to ask for combinatorial proofs of these results. The enumerative theory of continued fraction started with the influential work of Flajolet~\cite{flajolet-1980}. Many results follow as special cases of the following two results.
Let $x$ and $\rho_n, \tau_n$ for $n \geq 0$ be commuting indeterminates and 
let $\rho = (\rho_n)_{n \geq 0}$ and $\tau = (\tau_n)_{n \geq 0}$.
Then the \emph{Jacobi continued fraction}, or \emph{J-fraction} for short, is given by
\begin{equation}
\label{jfraction}
J(x; \tau, \rho) = \frac{1}{1 - \tau_0 x - 
\frac{\ds \rho_0 x^2}{\ds 1 - \tau_1 x - \frac{\rho_1 x^2}{\ddots}}}.
\end{equation}
A special case of this is the \emph{Stieltjes continued fraction}, or \emph{S-fraction} for short, which is
\begin{equation}
\label{sfraction}
S(x; \rho) = \frac{1}{1 - \frac{\ds \rho_0 x}{\ds 1 - 
\frac{\ds \rho_1 x^2}{\ddots}}}.
\end{equation}
Recall that a \emph{Motzkin path} is a path in $\mathbb{Z}^2$ using steps $\{(1, 1), (1, 0), (1, -1)\}$ starting from the origin, ending on the $x$-axis and staying in the first quadrant throughout. Similarly, a \emph{Dyck path} satisfies the same conditions as above, except that the set of steps is $\{(1, 1), (1, -1)\}$. Consider weighted Motzkin paths where northeast steps $(1, 1)$ get weight $\rho_i$ if they start at row $i$ and east steps $(1, 0)$ get weight $\tau_i$ if they stay at row $i$. Similarly, steps in weighted Dyck paths get weight $\rho_i$ as above. Define the weight of a path to be the product of its step weights times $x$ raised to its length.

\begin{theorem}[{\cite[Theorem 1]{flajolet-1980}}]
\leavevmode
\begin{enumerate}
\item The generating function for weighted Motzkin paths is the J-fraction $J(x; \tau, \rho)$.

\item The generating function for weighted Dyck paths is the S-fraction $S(x; \rho)$.

\end{enumerate}
\end{theorem}

See Goulden--Jackson~\cite[Chapter~5]{goulden-jackson-1983} and Viennot~\cite{viennot-1985} for many applications of these results.

Recall that $s_n$ denote the sum of the entries of the character table of $S_n$. Let $\mathcal{S}(x)$ be the (ordinary) generating function of the sequence $(s_n)$, i.e.
\begin{equation}
\label{def-S}
\mathcal{S}(x) = \sum_{n \geq 0} s_n x^n.
\end{equation}
To give a formula for $\mathcal{S}(x)$, we will need
the following generating functions. Recall that an \emph{involution} in $S_n$ is a permutation $w$ which squares to the identity.
Let $i_n$ be the number of involutions in $S_n$.
A special case of the above result, also due to Flajolet~\cite[Theorem 2(iia)]{flajolet-1980} gives  a continued fraction expansion of the generating function $\mathcal{I}(x)$ of involutions in $S_n$ as the J-fraction
\begin{equation}
\label{formula-I}
\mathcal{I}(x) = \sum_{n \geq 0} i_n x^n =
J(x; (1), (n+1)) = 
\frac{1}{\ds 1 - x - \frac{x^2}{\ds 1 - x - \frac{2x^2}{\ddots}}}.
\end{equation}
Flajolet also showed in the same theorem~\cite[Theorem 2(iib)]{flajolet-1980} that the generating function of odd double factorials is the S-fraction,
\begin{equation}
\label{formula-D}
\mathcal{D}(x) = \sum_{n \geq 0} (2n-1)!! \, x^n=
S(x; (n+1)) = 
\frac{1}{\ds 1 - \frac{x}{\ds 1 - \frac{2x}{\ddots}}}.
\end{equation}
The quantity $o_r(n)$ (defined in \cref{num odd cycle}) and its generalizations have been studied in \cite{leanos-moreno-rivera-2012}.
Setting $t=0, n=0$ and $u_1=1$ in the same theorem ~\cite[Theorem 2]{flajolet-1980} we obtain the generating function for $o_r(n)$ as the J-fraction
\begin{equation}
\label{formula-R}
\mathcal{R}_r(x) = \sum_{n \geq 0} o_r(n) x^n=
J(x; (1), (r(n+1))) = 
\frac{1}{\ds 1 - x - \frac{r x^2}{\ds 1 - x - \frac{2 r x^2}{\ddots}}}.
\end{equation}

Let $x,x_1,x_2,\ldots$ be a family of commuting indeterminates. The following result answers a question of Amdeberhan~\cite{Amderfan}.

\begin{theorem}
\label{thm:gen-fn-sn}
The number of square roots of a permutation with cycle type $\lambda$ written in frequency notation as in \eqref{freq not}, which we denoted $\Gamma_\lambda$, is the coefficient of $x_1^{m_1} x_2^{m_2} \dots x_n^{m_n}$ in
\[
\prod_{i \geq 1} \mathcal{D}(2i x_{2i}^{2}) \mathcal{R}_{2i-1} (x_{2i-1}).
\]
Consequently, the generating function of the character table sum is 
\[
\mathcal{S}(x) = \prod_{i \geq 1} \mathcal{D}(2ix^{4i}) \mathcal{R}_{2i-1} (x^{2i-1}).
\]
\end{theorem} 

\begin{proof}
By \cref{prop:colsum-lambda}, the multiplicities of even parts must be even. Therefore, we write $\lambda$ as
$\lambda = \langle 1^{m_1}, 2^{2m_2}, 3^{m_3}, 4^{2m_4}, \dots \rangle$.
Thus
\begin{multline*}
\sum_\lambda \Gamma_{\lambda} 
x_1^{m_1} x_2^{2m_2} \dots  
 = \sum_{m_1 = 0}^\infty \sum_{m_2 = 0}^\infty \cdots
\prod_{i = 1}^{\infty} (2m_{2i} - 1)!! \, (2i)^{m_{2i}} 
x_{2i}^{2m_{2i}}
\prod_{j = 0}^{\infty} o_{2j+1}(m_{2j+1}) x_{2j+1}^{m_{2j+1}} \\
=
\left( \sum_{m_1 = 0}^\infty o_{1}(m_{1}) x_{1}^{m_1} \right)
\left(\sum_{m_2 = 0}^\infty (2m_2 - 1)!! 2^{m_{2}} x_2^{2m_{2}} \right)
\cdots .
\end{multline*}
We then use \eqref{formula-D} and \eqref{formula-R} to prove the first statement.
Now replace $x_i$ by $x^i$, where $x$ is an indeterminate. Then the coefficient of $x_n$ is precisely the sum of columns in the character table of $S_n$.
\end{proof}

\subsection{Proof of \cref{property:S}}
\label{sec:Sn asymp}
The sequence $(i_n)$ satisfies the well-known recurrence relation for $n\geq 2$ given by
\begin{equation}
\label{recurrenceA}
i_n=i_{n-1} + (n-1)i_{n-2},
\end{equation}
with $i_0=i_1=1$.

\begin{lemma}
\label{lem:invol-second}
For $n\geq 2$, we have $2i_{n-1} \leq i_n \leq n \ i_{n-1}$.
\end{lemma}

\begin{proof}
Since $i_1 = 1$, $i_2 = 2$ and $i_3 = 4$, the inequalities are valid for $n = 2, 3$.
For $n\geq 4$, we have $i_n \geq \sqrt{n} i_{n-1}\geq 2 i_{n-1}$, 
where the  
first inequality follows from \cite[Lemma 2]{chowla-herstein-moore-1951} and the second is obvious.
We prove the other inequality by induction on $n$. This inequality is valid for $n=3$, as mentioned above. Using \eqref{recurrenceA}, we have 
by the induction hypothesis, 
\begin{align*}
		i_n 
		& \leq   (n-1)i_{n-2} + (n-1)i_{n-2}  \\
		& = 2(n-1) i_{n-2} 
		 \leq  (n-1)	i_{n-1} \leq ni_{n-1},
\end{align*}
where the second line uses the first part of the statement. This completes the proof. 
\end{proof}

Recall that \emph{derangements} are permutations without fixed points.
We define another sequence 
$(g_n)$ by 
\[
g_n:=\sum_{\substack{\lambda \vdash n \\ m_1(\lambda) = 0}} \Gamma_{\lambda}, \quad n \geq 1 \quad \text{and} \quad g_0=1.
\]
Then $g_n$ counts the sum of those columns of the character table of
$S_n$ which are indexed by the conjugacy classes corresponding to  derangements.

\begin{lemma}
\label{lem:typea-asymp-main}
For $n \geq 4$, we have 
$i_kg_{n-k} \leq {i_{n-1}}/(n-2) $ for all $0 \leq k \leq n-3.$ 
\end{lemma}

\begin{proof}
We use induction on $n$. The statement holds for $n=4$, as we can verify using $i_3=4, i_2=2, i_1=1, i_0=1, g_3=1, g_2=0, g_1=0, g_0=0$. We assume the statement to be true for $n$ and prove this for $n + 1$. 
When $k=n-2$, we have $i_{n-2}g_3=i_{n-2}	\leq i_n/(n-1)$ by \eqref{recurrenceA}. For $ k \leq n-3$, by using \eqref{recurrenceA}, the left hand side becomes
\[
i_kg_{n+1-k}  =  (i_{k-1}+(k-1)i_{k-2} \big) g_{n-(k-1)}=  i_{k-1} g_{n-(k-1)} +(k-1) i_{k-2} g_{n-1-(k-2)}. \] 
By induction, we have
\[ 
i_kg_{n+1-k} \leq  \frac {i_{n-1}}{n-2} + \frac{k-1}{n-3}i_{n-2}. 
\]
By \cref{lem:invol-second}, we have $i_{n-1} \leq (n-1)i_{n-2}$, we so we get
\[
(n-3)	i_{n-1}  \leq  (n-2)(n-1) i_{n-2}\leq (n-2)(n-1)
\big( (n-3)-(k-1) \big) i_{n-2}, 
	\]
where the last inequality follows as $(n-3)-(k-1)\geq 1$. We can rewrite this inequality as
\[
(n-3)i_{n-1} + (k-1)(n-1)(n-2)i_{n-2} \leq (n-1) (n-2)(n-3)i_{n-2}. 
\]
This implies, by adding $(n-3)(n-2)i_{n-1}$ to both sides, that
\[(
n-3)(n-1)i_{n-1} + (k-1)(n-1)(n-2)i_{n-2} \leq (n-1) (n-2)(n-3)i_{n-2} + (n-3)(n-2)i_{n-1}.
\]
Now, divide each term by $(n-1)(n-2)(n-3)$ to obtain
\begin{equation}
\label{eqA}
\frac{i_{n-1}}{n-2} + \frac{k-1}{n-3}i_{n-2} \leq i_{n-2} +\frac{1}{n-1}i_{n-1}=\frac{i_n}{n-1}. 
\end{equation}
This completes the proof. 
\end{proof}

The next result is a convolution-type statement involving $s_n, g_n$ and $i_n$. 
\begin{proposition}
\label{prop:type-a-convol} 
For a positive integer $n$, we have 
\[
s_n=  \sum _{k=0}^n i_k g_{n-k}.
\]
\end{proposition}

\begin{proof}
We have 
\[
s_n= \sum_{\lambda \vdash n } \Gamma_{\lambda} 
= \sum _{k=0}^n \sum_{\substack{\lambda \vdash n \\ m_1(\lambda)=k} }
 \Gamma_{\lambda}.
\]
Given a partition $\lambda=\langle 1^{k}, 2^{m_2}, \dots, n^{m_n}\rangle$, let $\nu= \langle 2^{m_2}, \dots, n^{m_n} \rangle$, so that  $\lambda=\langle 1^k\rangle \cup \nu$. Recall that $w_{\lambda}$ is a permutation with cycle type $\lambda$. Then observe that
\[
\{x\in S_n \mid x^2=w_{\lambda}\}=\{y\in S_k \mid y^2= e \}
\times \{z\in S_{\{k+1,\ldots,n\}} \mid z^2=w_{\nu}\},
\]
where $S_{ \{ k+1, \dots, n \} } $ denotes the group 
of bijections on the set $\{k+1, \dots, n\}$.
Finally, 
\[
\sum_{k=0}^n\sum_{\substack{\lambda \vdash n \\ m_1(\lambda)=k}} \Gamma_{\lambda}
=\sum_{k=0}^{n}i_k\sum_{\substack{\nu \vdash n-k \\ m_1(\nu)=0}}\Gamma_{\nu}=\sum_{k=0}^n i_k g_{n-k}.
\]
This completes the proof. 
\end{proof}
 
\begin{theorem}
\label{thm:bothsidebounds}
\cref{property:S} holds for all symmetric groups.  
\end{theorem}

\begin{proof}
It is easy to check the property for $S_1, S_2$ and $S_3$.
We shall prove a more stronger result, namely for $n\geq 4$,
\begin{equation}
\label{eq:snbounds}
s_n \leq i_{n}+i_{n-1}.
\end{equation}
By \cref{prop:type-a-convol}, 
we have 
\[	s_n =  \sum _{k=0}^n i_k g_{n-k} = \sum 
	_{k=0}^{n-3} i_k g_{n-k} + i_{n-2}g_2+i_{n-1} g_1+i_n.  \]
Using \cref{lem:typea-asymp-main} and the fact that $g_0 = 1, g_1=g_2=0$, we get
\[ s_n  \leq (n-2) \frac{i_{n-1}} {n-2} +i_n  = i_{n-1} + i_n.\] 
Thus, $s_n  \leq 2i_n$ as $i_{n-1}\leq i_n$. This completes the proof. 
\end{proof}

We now confirm the observation of user Lucia~\cite{Amderfan}. 

\begin{corollary}
\label{cor:asymp-sn-chartable}
The total sum sequence $(s_n)$ grows asymptotically as fast as $(i_n)$ and hence
\[
s_n \sim \left( \frac{n}{e} \right)^{n/2} \frac{e^{\sqrt{n}-1/4}}{ \sqrt{2} }.
\]
\end{corollary}
	
\begin{proof}
By \cite[Lemma 2]{chowla-herstein-moore-1951}, we have 	
\[\frac{i_{n-1}}{i_n} \leq \frac{1}{\sqrt{n}}.\]
Therefore, from the proof of \cref{thm:bothsidebounds}, we have 
\[ 1 \leq \frac{s_n}{i_n} \leq 1+ \frac{ i_{n-1}}{i_n} \leq 1+ \frac{1}{\sqrt{n}}. \]
Now, both the sequences $(1)$ and $(1+1/\sqrt{n})$ converges to $1$. Hence, by the sandwich theorem, the sequence $s_n/i_n$ converges to $1$.
Thus, the sequence $(s_n)$ grows asymptotically as fast as the sequence $(i_n)$. 
The asymptotics of $(i_n)$ is derived by Chowla--Herstein--Moore~\cite[Theorem 8]{chowla-herstein-moore-1951} and is stated above.
This completes the proof. 
\end{proof}
	
\section{Weyl groups of type B}
\label{sec:Bn}	

Regard $S_{2n}$ as the group of permutations of the set $\{\pm 1,\ldots, \pm n\}$. For an integer $n\geq 1$, the group $B_n$ is defined as 
\begin{equation*}
B_n:=\{w \in S_{2n} \, \mid \, w(i) + w(-i)=0, \text{ for all } i,\,  1\leq i \leq n \}.
\end{equation*}
The group $B_n$ is known as the \emph{hyperoctahedral group} or the \emph{group of signed permutations}.

\begin{remark}
Instead of treating $B_n$ as a subgroup of $S_{2n}$, one can also define it as the wreath product, $B_n=\Z_2 \wr S_n$. In \cref{sec:wreath}, we will describe our results for wreath products $\Z_r \wr S_n$, where $r \in \mathbb{N}_{\geq 1}$. Here we have stated the results in the standard terminology for $B_n$ for the benefit of the reader.
\end{remark}

The following facts can be found in the book by Musili~\cite{musili-1993}.
Every element $w \in B_n$ can be uniquely expressed as a product of cycles
\begin{equation*}
\label{eq:cycle decomposition}
w=u_1\overline {u}_1 \cdots u_r \overline{u}_r v_1\cdots v_s,
\end{equation*} 
where  for $1\leq j\leq r$, $u_j \overline{u}_j=(a_1, \ldots, a_{\lambda_j})(-a_1, \ldots, -a_{\lambda_j})$ for some positive integer $\lambda_j$ and for $1 \leq k \leq s$, $v_k=(b_1, \ldots, b_{\mu_k},-b_1,\ldots, -b_{\mu_k})$ for some positive integer $\mu_k$. The element $u_j\overline{u}_j$ is called a \emph{positive cycle} of length $\lambda_j$ and $v_k$ is called a \emph{negative cycle} of length $\mu_k$. 
This cycle decomposition of $w$ determines a unique pair of partitions $(\lambda \mid \mu)$, where $\lambda=(\lambda_1,\ldots, \lambda_r)$,  $\mu=(\mu_1,\ldots, \mu_s)$ and $|\lambda| + |\mu| = n$. The pair $(\lambda \mid \mu)$, also known as a \emph{bipartition}, is the \emph{cycle type} of $w$. We write $(\lambda \mid \mu) \models n$ for such a bipartition of $n$.

\begin{theorem}[{\cite[Theorem 7.2.5]{musili-1993}}]
The set of conjugacy classes of $B_n$ is in natural bijection with the set of bipartitions $(\lambda \mid \mu) \models n$.
\end{theorem}

Let $\Gamma^B_{(\lambda \mid \mu)}$ be the column sum corresponding to the conjugacy class $(\lambda \mid \mu) \models n$ 
in the character table of $B_n$. By \cref{thm:FS}, $\Gamma^B_{(\lambda \mid \mu)}$ is the number of square roots of an element of cycle type $(\lambda \mid \mu)$.

\subsection{Generating functions}

The following results follow immediately from an analysis of the cycle type of a signed permutation.

\begin{lemma}
	\label{lem:square-Dn}
	\begin{enumerate}
		\item The square of a positive or negative cycle of odd length is a positive cycle of the same length.
		
		\item The square of an even positive (resp. negative) cycle is a product of two positive (resp. negative) cycles of half that length.
	\end{enumerate}
\end{lemma}

\begin{proposition}
	\label{prop:non-square-B}
	The bipartition $(\lambda \mid \mu) \models n$ is the cycle type of a square element of $B_n$ if and only if each even part of $\lambda$ has even multiplicity and each part of $\mu$ has even multiplicity.
\end{proposition}

We now extend the results in \cref{cor:bes-ols} and \cref{prop:bes-ols} to $B_n$. 

\begin{corollary}
The number of columns of the character table of $B_n$ that have zero sum (i.e. the number of bipartitions $(\lambda \mid \mu)$ such that $\Gamma^B_{(\lambda \mid \mu)} = 0$) is the same as the number of bipartitions $(\lambda \mid \mu) $ of $n$ such that at least one part of $\lambda$ is congruent to $2 \pmod 4$ or $\mu$ has at least one odd part.
\end{corollary}

The next result is a special case of \cref{thm:abs-sqr-gen-fun}.

\begin{theorem}
The generating function for the number of columns in the character table of $B_n$ with nonzero sum is
\[
\prod_{j=1}^{\infty}\frac{1}{(1-q^{2j-1}) (1-q^{2j}) (1-q^{4j})}.
\]
\end{theorem}

Let $s_n^B$ be the sum of the entries of the character table of $B_n$.
Let $\mathcal{S}^B(x)$ be the ordinary generating function of $(s_n^B)$. Recall the functions $\mathcal{D}(x)$ from \eqref{formula-D} and $\mathcal{R}_r(x)$ from \eqref{formula-R}. Let $\{x_1, x_2, \dots, y_1, y_2, \dots \}$ be a doubly infinite family of commuting indeterminates.
The following result is a special case of both \cref{thm:gen-fun-sqrt} and \cref{thm:gen-fun-abs-sqrt}.

\begin{theorem}
\label{thm:gen-fun-sqrt-bn}
Write each partition in the bipartition $(\lambda \mid \mu) \models n$ in frequency notation, with 
$\lambda = \langle 1^{m_{1}}, 2^{m_{2}}, \dots, \allowbreak n^{m_{n}}  \rangle$ and
$\mu = \langle 1^{p_{1}}, 2^{p_{2}}, \dots, n^{p_{n}}  \rangle$.
Then $\Gamma^B_{(\lambda \mid \mu)}$ is the coefficient of 
$(x_{1}^{m_{1}} \dots x_{n}^{m_{n}})
(y_{1}^{p_{1}} \dots y_{n}^{p_{n}})$
in
\[
\ds \prod_{k \geq 1} \Big( \ds 
\mathcal{D}(4k x_{2k}^2) \, \mathcal{D}(4k y_{2k}^2)
\mathcal{D}(2(2k-1) y_{2k-1}^{2})  \,
\mathcal{R}_{\frac{(2k-1)}{2}} (2x_{2k-1}) \Big).
\]
Therefore, 
\[
\ds \mathcal{S}^B(x) = \prod_{k \geq 1} 
\Big( \ds \mathcal{D}(4k x^{4k}) \,
\mathcal{D}(2kx^{2k}) \,
\mathcal{R}_{\frac{(2k-1)}{2}} (2x^{2k-1}) \Big).
\]
\end{theorem}

\subsection{Proof of \cref{property:S}}
Let $i_n^B$ be the number of involutions in $B_n$. 
The first nine terms of $(i_n^B)$ are given by \cite[Sequence A000898]{oeis}
\[
2,6,20,76,312,1384,6512,32400,168992.
\]
The sequence $(i_n^B)$ satisfies the recurrence relation~\cite[Theorem 2.1]{Chow06}
\begin{equation}
\label{recurrence-invol-B}
i_n^B=2i_{n-1}^B + 2(n-1)i_{n-2}^B, \quad n\geq 2,
\end{equation}
with $i_0^B=1$ and $i_1^B=2$.

\begin{lemma}
\label{lem:type-b-invol-ratio}
For positive integers $n$, we have 
\[ 
\sqrt{2n} \leq \frac{i_n^B}{i_{n-1}^B} \leq \sqrt{2n}+2. 
\]
\end{lemma}

\begin{proof}
The proof is by induction on $n$. When $n=1$, $\sqrt{2} \leq i_1^B/i_{0}^B = 2 \leq 2 + \sqrt{2}$. So, the result is correct. We assume the result for $n$. 
Since $i_{n+1}^B=2i_{n}^B + 2n i_{n-1}^B$ by \eqref{recurrence-invol-B}, we have 
\[
\frac{i_{n+1}^B}{i_n^B} = 2+ \frac{2n}{i_n^B/i_{n-1}^B} \leq 2+ \frac{2n}{\sqrt{2n}} \leq 2+\sqrt{2(n+1)}. 
\]
On the other side, we have 
\[
\frac{i_{n+1}^B}{i_n^B} = 2+ \frac{2n}{i_{n}^B / i_{n-1}^B} 
\geq 2+ \frac{2n}{\sqrt{2n}+2} 
\geq 2+ \frac{(\sqrt{2n+2}+2) (\sqrt{2n+2}-2)}{\sqrt{2n}+2} 
\geq \sqrt{2n+2}. 
\]
This proves the inequality for $n+1$ and completes the proof. 			
\end{proof}

\begin{lemma}
\label{lem:invol-first-b}
For positive integers $n\geq 5$, we have $ 4i_{n-1}^B \leq i_n^B \leq n \ i_{n-1}^B.$ 
\end{lemma}

\begin{proof}
We proceed by induction on $n$. When $n=5$, this can be verified. By using \eqref{recurrence-invol-B} repeatedly, we get
\begin{multline*}
i_n^B   =  2i_{n-1}^B+2(n-1)i_{n-2}^B 
=  2 (2i_{n-2}^B+2(n-2)i_{n-3}^B) + 2(n-1)i_{n-2}^B 
= (2n + 2) i_{n-2}^B + (4n-8) i_{n-3}^B.
\end{multline*}
Subsequently, we rewrite and use induction to get 
\begin{multline*}
i_n^B   =   8i_{n-2}^B+(2n-6)i_{n-2}^B+4(n-2)i_{n-3}^B \\
 \geq  8 i_{n-2}^B + 4(2n-6)i_{n-3}^B+4(n-2)i_{n-3}^B.
 =  8 i_{n-2}^B + 4(3n-8)i_{n-3}^B  
\end{multline*}
As $n \geq 6$, we finally have 	
\[ 
i_n^B \geq 8 i_{n-2}^B + (8n-8)i_{n-3}^B = 4i_{n-1}^B. 
\]
The proof of the other inequality directly follows by using \cref{lem:type-b-invol-ratio} and observing that for $n \geq 6$, we have $\sqrt{2n} +2 \leq n.$ 
\end{proof} 

The first nine values of the sequence $(s_n^B)$ are
\[
2,8,26,112,410,1860, 8074,40376,199050.
\]
To find the asymptotics of $s_n^B$, we define
\begin{equation}
	\label{def:g function in B}
	g_n^B:=\sum_{\substack{(\lambda \mid \mu) \models n \\ m_1(\lambda) = 0}}  \Gamma^B_{(\lambda \mid \mu)}.
\end{equation}

\begin{lemma}
\label{lem:typea-asymp-main-b}
For positive integers $n \geq 6$, we have 
$i_k^Bg_{n-k}^B\leq  2i_{n-1}^B /(n-2) $ for all $0 \leq k \leq n-3.$ 
\end{lemma}

\begin{proof}
Here also, we proceed by induction on $n$. When $n=6$, we can verify that the statement holds. We assume the statement to be true for $n$ and prove this for $n+1$. 
For $ k \leq n-3$, we have 
\[
i_k^Bg_{n+1-k}^B   =  \big  (2i_{k-1}^B+2(k-1)i_{k-2}^B \big) g_{n-(k-1)}^B   
=  2i_{k-1}^B g_{n-(k-1)}^B +2(k-1) i_{k-2}^B g_{n-1-(k-2)}^B. 
\]
By the induction hypothesis, we now have 
\[
i_k^Bg_{n+1-k}^B  \leq   \frac {4i_{n-1}^B}{n-2} + (k-1) \frac{4i_{n-2}^B}{n-3}. 
\]
For positive integers $n$, we have $i_{n-1}^B \leq (n-1)i_{n-2}^B$ by \cref{lem:invol-first-b}.
So, we get
\[
(n-3)	i_{n-1}^B  \leq  (n-2)(n-1) i_{n-2}^B \leq (n-2)(n-1)[(n-3)-(k-1)]i_{n-2}^B,
\]
where the last inequality follows because $(n-3)-(k-1)\geq 1$. We can rewrite it as
\[
(n-3)i_{n-1}^B + (k-1)(n-1)(n-2)i_{n-2}^B \leq (n-1) (n-2)(n-3)i_{n-2}^B. 
\]
This implies, by adding $(n-3)(n-2)i_{n-1}^B$ to both sides, that
\[
(n-3)(n-1)i_{n-1}^B + (k-1)(n-1)(n-2)i_{n-2}^B \leq (n-1) (n-2)(n-3)i_{n-2}^B + (n-3)(n-2)i_{n-1}^B.
\]
Dividing each side by $(n-1)(n-2)(n-3)$ and multiplying by $4$, we get
\[	 
\frac {4i_{n-1}^B}{n-2} + (k-1) \frac{4i_{n-2}^B}{n-3}  \leq  \frac{4i_{n-1}^B	}{n-1}+4i_{n-2}^B .
\]
Therefore, we have by \eqref{recurrence-invol-B},
\[ 
i_k^Bg_{n+1-k}^B \leq  \frac{2i_n^B}{n-1}. 
\]
Thus for $ k \leq n-3$, our statement holds. When $k=n-2$, we have 
\[
i_{n-2}^Bg_{n+1-(n-2)}^B = i_{n-2}^Bg_3^B=2i_{n-2}^B\leq \frac{2i_n^B}{n-1},
\] 
completing the proof.		  
\end{proof} 

Here, we have the following counterpart of \cref{prop:type-a-convol}. 

\begin{proposition}
\label{lem:type-b-convol} 
For positive integers $n$, we have 
\[s_n^B=  \sum _{k=0}^n i_k^B g_{n-k}^B.\]	
\end{proposition}

\begin{proof}
We have 
\begin{equation}
\label{eq:conv}
s_n^B  
= \sum_{(\lambda \mid \mu) \models n} \Gamma^B_{(\lambda \mid \mu)} 
=  \sum _{k=0}^n \sum_{\substack{(\lambda \mid \mu) \models n \\ m_1(\lambda)=k}} \Gamma^B_{(\lambda \mid \mu)}.
\end{equation}
Given $\lambda=\langle 1^k,2^{m_2}, \dots, n^{m_n} \rangle$, 
let $\nu=\langle 2^{m_2}, \dots, n^{m_n} \rangle$. 
Then we have 
$ (\lambda \mid \mu)= ( \langle 1^k \rangle \mid \emptyset) \cup (\nu \mid \mu)$.
Let $w_{(\lambda \mid \mu)}$ 
denote a signed permutation with cycle type $(\lambda \mid \mu)$.
Let $B_{\{k+1, \dots, n\}}$ be the group of bijections $w$ on the set $\{\pm (k+1), \dots, \pm n \}$ such that $w(i)+w(-i)=0$ for all $k+1 \leq i \leq n$.
Observe that 
\[
\{x \in B_n \mid x^2=w_{(\lambda \mid \mu)}\}=\{y\in B_k \mid y^2= e\}
\times \{z\in B_{\{k+1,\ldots,n\}} \mid z^2=w_{(\nu \mid \mu)}\}.
\]
Therefore,
 $\Gamma^B_{(\lambda \mid \mu)}
 =\Gamma^B_{( \langle 1^k \rangle  \mid \emptyset)} \times \Gamma^B_{(\nu \mid \mu)}$.
Thus, the right hand side of \eqref{eq:conv} is reduced to
\[
 \sum_{k=0}^n	\Gamma^B_{( \langle 1^k \rangle \mid  \emptyset)}
  \sum_{\substack{(\nu \mid \mu) \models n-k \\ m_1(\nu)=0} } 
 \Gamma^B_{(\nu \mid \mu)} 
 = \sum_{k=0}^n i_k^B g_{n-k}^B. 
\]
\end{proof}

\begin{theorem}
\label{thm:bothsidebounds-b} 
\cref{property:S} holds for all hyperoctahedral groups $B_n$.  
\end{theorem}

\begin{proof}
For $n\leq 5$, this can be directly verified. For every positive integer $n \geq 6$, we shall prove the following:
\begin{equation}
\label{eqn:bnbounds}
s_n^B \leq i_{n}^B+2i_{n-1}^B+2i_{n-2}^B. 
\end{equation} 
By \cref{lem:type-b-convol}, we have 
\[
s_n^B  =  \sum _{k=0}^n i_k^B g_{n-k}^B =  i_n^B + i_{n-1}^B g_1^B +i_{n-2}^Bg_2^B + \sum 
_{k=0}^{n-3} i_k ^Bg_{n-k}^B
\]
It is easy to see that $g_1^B=0$ and $g_2^B=2$. We now use \cref{lem:typea-asymp-main-b} to get 
\[ s_n^B \leq  i_n^B + 2i_{n-2}^B+(n-2) \frac{2i_{n-1}^B} {n-2}. 
\]
By \cref{lem:invol-first-b}, the proof is now complete for $n \geq 6$. 
\end{proof}

\begin{corollary}
\label{cor:asymp-B}
The total sum sequence $(s_n^B)$ grows asymptotically as fast as $(i_n^B)$ and hence
\[
s_n^B \sim   \frac{e^{\sqrt{2n}}}{ \sqrt{2e} }
 \left(\frac{2n}{e} \right)^{n/2}\]
\end{corollary}	

\begin{proof}
By \cref{lem:type-b-invol-ratio}, we have 	
\[\frac{i_{n-1}^B}{i_n^B} \leq \frac{1}{\sqrt{2n}}.\]
Therefore, from the proof of \cref{thm:bothsidebounds-b}, we have 
\[ 1 \leq \frac{s_n^B}{i_n^B} \leq 1+ \frac{ 2i_{n-2}^B}{i_n^B} + \frac{2i_{n-1}^B}{i_n^B} \leq 1+ \frac{4}{\sqrt{2n}}. \]
Now, both the sequences $(1)$ and $(1+4/\sqrt{2n})$ converges to $1$. Hence, by the sandwich theorem, the sequence $s_n^B/i_n^B$ converges to $1$.
Thus, the sequence $(s_n^B)$ grows asymptotically as fast as the sequence $(i_n^B)$. 
By using a result of Lin~\cite[Eq. (5)]{lin-2013}, 
we have the asymptotics of $i^B_n$.
This completes the proof. 	
\end{proof}

\section{Weyl groups of type D} 
\label{sec:Dn}
	
The Weyl group of type $D$, also known as the \emph{demihyperoctahedral group} $D_n$, is defined as the following index two subgroup of $B_n$: 
\[
D_n:=\{w\in B_n \mid w(1)\cdots w(n)>0\}.
\]

\begin{lemma}
\label{lem:conjug-class-Dn}
Let $w\in B_n$ have cycle type $(\lambda \mid \mu)$. Then $w\in D_n$ if and only if the length $\ell(\mu)$ is even. 
\end{lemma}

For a proof, take $r=q=2$ in \cite[Lemma 2.3]{mishra-paul-singla-2023}. 
A bipartition $(\lambda \mid \mu) \models n$ with $\ell(\mu)$ is even, is called \emph{split} if the associated conjugacy class $C_{(\lambda \mid \mu)}$ in $B_n$ splits into two conjugacy classes in $D_n$. Otherwise, $(\lambda \mid \mu)$ is called \emph{non-split}.
The following result characterizes split and non-split bipartitions of $n$.

\begin{proposition}[{\cite[Theorem 8.2.1]{musili-1993}}]
\label{prop:Conj-class-cplit-in-Dn} 
Suppose $(\lambda \mid \mu)$ is a bipartition of $n$ with $\ell(\mu)$ even, and $\pi\in D_n$ has cycle type $(\lambda \mid \mu)$. Then the associated conjugacy class $C_{(\lambda \mid \mu)}$ in $B_n$ splits into two conjugacy classes in $D_n$ if and only if $\mu=\emptyset$ and all the parts of $\lambda$ are even. Equivalently, the class $C_{(\lambda \mid \mu)}$ remains a $D_n$ conjugacy class if and only if either $\mu \neq \emptyset$ or else one of the parts of $\lambda$ is odd. In particular, for odd $n$, no conjugacy class of $B_n$ splits.
\end{proposition}

\begin{corollary}
\label{cor:conjug-class-Dn}
Conjugacy classes in $D_n$ are indexed by bipartitions $(\lambda \mid \mu)$ of $n$ where
\begin{itemize}
\item $\ell(\mu)$ is even, and
\item if $\mu = \emptyset$ and $\lambda$ contains only even parts, then there are two kinds of bipartitions $(\lambda_+ \mid \emptyset)$ and $(\lambda_-\mid \emptyset)$.
\end{itemize}
\end{corollary}

\subsection{Generating functions}

In this section, we give generating functions for the number of columns in the character table of $D_n$ with nonzero sum as well as the total sum of the character table.
We begin by characterizing the existence of square roots of elements in $D_n$ in terms of their cycle type.

\begin{proposition}
\label{prop:square-criterion-Dn} 
A bipartition $(\lambda \mid \mu) \models n$ with $\ell(\mu)$ even is the cycle type of a square element of $D_n$ if and only if  the following holds: 
\begin{enumerate}
\item all even parts of $\lambda$ have even multiplicity,
\item all parts of $\mu$ have even multiplicity, and 
\item either $\lambda$ has an odd part or $4\mid \ell(\mu)$. 
\end{enumerate}

\end{proposition}
\begin{proof}
Let $\omega \in D_n$ have cycle type $(\lambda \mid \mu)$. By \cref{prop:non-square-B}, $\omega$ has a square root in $B_n$ if and only if (1) and (2) hold. Assuming $\omega$ has a square root in $B_n$, it is enough to show that  $\omega$ has a square root in $D_n$ if and only if either $\lambda$ has a odd part or $4\mid \ell(\mu)$.

Suppose $\pi \in D_n$ such that $\pi^2=\omega$. If $\lambda$ has an odd part, we are done. If not, observe that $\pi$ cannot have a negative cycle of odd length. This is because, by \cref{lem:square-Dn}, the square of an odd negative cycle is a positive cycle of the same length. 
This is impossible as $\lambda$ has no odd part.
 
Since $\pi \in D_n$, it has an even number of negative cycles. By the last argument, $\pi$ has even number of negative cycles of even length, and  by squaring each of them, we get two negative cycles, concluding the result $4 \mid \ell(\mu)$.

For the converse, note that all parts of $\mu$ has even multiplicity, by  (2). By \cref{lem:square-Dn}, each pair of negative cycles of length $t$ in $\omega$ can be obtained  only by squaring a negative cycle of length $2t$. Therefore, any square root (in $B_n$) of $\omega$ must have at least $\ell(\mu)/2$ negative cycles. Thus there exists an element $\pi\in B_n$ such that its positive cycles squares to positive cycles of $\omega$ and its negative cycles squares to negative cycles of $\omega$. Then, the number of negative cycles of $\pi$ is exactly $\ell(\mu)/2$, which is even if $4\mid \ell(\mu)$. 
 
If  $4\nmid \ell(\mu)$, $\lambda$ has at least one odd part by (3). Note that the element $\pi$ constructed in the last paragraph does not belongs to $D_n$ as $\ell(\mu)/2$ is odd. But, by \cref{lem:square-Dn}, an odd length positive cycle can also be obtained by taking the square of a negative cycle of the same length. Since $\lambda$ has at least one odd part, we can form a square root of $\omega$ by choosing a negative square root for that odd part and proceeding as before for the other parts. The square root so obtained must belongs to $D_n$ as it has $\ell(\mu)/2+1$  many negative cycles.
This concludes the proof.  
\end{proof}

Using \cref{prop:square-criterion-Dn}, we next obtain the generating function for the number of conjugacy classes in $D_n$ with non-zero column sums. The sequence of the number of partitions of a positive integer $n$ into an even number of parts is given in \cite[Sequence~A027187]{oeis}.

\begin{lemma}[{\cite[Example 7, Page 39]{Fine88}}]
\label{lem:type-d-lem-colsum-nonzero-gf-aux}
The generating function for the number of partitions of a positive integer $n$ with an even number of parts is 
\[
\frac{1}{2} \left( 	\prod_{j=1}^{\infty}\frac{1}{1-q^{j}}+
\prod_{k=1}^{\infty}\frac{1}{1+q^{k}}
\right).
\] 	
\end{lemma}

Let $\mathcal{B}_n$ be the set of bipartitions of $n$.
The following result extend \cref{prop:bes-ols} to the group $D_n$.

\begin{theorem}
\label{thm:gen-fun-sqrt-dn} 
The generating function for the number of conjugacy classes in $D_n$ 
with non-zero column sum is
\begin{equation*}
\left(	\prod_{i=1}^{\infty}\frac{1}{1-q^{4i}} \right) \left(
\left( \prod_{j=1}^{\infty}\frac{1}{1-q^{2j}}\right) \left( \prod_{k=0}^{\infty}\frac{1}{1-q^{2k+1}}-1\right)
 + 
 \frac{1}{2}  
 \left( \prod_{j = 1}^{\infty}\frac{1}{1-q^{2j}}+\prod_{k = 1}^{\infty}\frac{1}{1+q^{2k}}\right) 
 + 1 \right)
 -1. 
\end{equation*}
\end{theorem}

\begin{proof}
Let 
\[ 
\B_n^1= \{(\lambda \mid  \mu) \models n \mid m_{2r}(\lambda)  \mbox { and } m_{r}(\mu) \mbox{ are even for each }r \mbox{ and } \lambda \mbox{ has at least one odd part}\},
\]
and 
\[ 
\B_n^2= \{(\lambda \mid  \mu) \models n \mid m_{2r}(\lambda)  \mbox { and } m_{r}(\mu) \mbox{ are even for each }r \mbox {, }\lambda \mbox { has no odd parts} \mbox { and } \ell(\mu) \equiv 0 \pmod 4\}. 
\] 
Note that $\B_n^1 \cap \B_n^2= \emptyset$. By \cref{prop:square-criterion-Dn}, a permutation 
$\pi \in D_n$ 
with cycle type $(\lambda \mid \mu)$ is the square of some element 
if and only if 
$\pi \in \B_n^1 \cup \B_n^2$. We now compute the generating functions for the sequences $(|\B_n^1|)$ and $(|\B_n^2|)$ separately. 

It is easy to see that $\B_n^1$ can be written as  
$\B_n^1=\B_n^3 \setminus \B_n^4$, where 
\[ 
\B_n^3 =  \{(\lambda \mid  \mu) \models n \mid m_{2r}(\lambda)  \mbox { and } m_{r}(\mu) \mbox{ are even for each }r \},
\]
and 
\[ 
\B_n^4 = \{(\lambda \mid  \mu) \models n \mid m_{2r}(\lambda)  \mbox { and } m_{r}(\mu) \mbox{ are even for each }r \mbox{ and } \lambda \mbox{ has no odd parts}\}.
\]
We next find generating functions for the sequences $(|\B_n^3|),(|\B_n^4|)$ and $(|\B_n^2|)$. To that end, we define the following three sets: 
\begin{align*}
\B_n^5 =& 
\{(\lambda \mid  \mu) \models n \mid m_{4r+2}(\lambda) =0 \mbox { and } m_{2r+1}(\mu)=0 \mbox{ for each }r \}, \\
\B_n^6 = &
\{ (\lambda \mid  \mu) \models n \mid m_{4r+2}(\lambda) =0, m_{2r+1}(\lambda)=0 \mbox { and } m_{2r+1}(\mu)=0 \mbox{ for each }r \}, \\
\B_n^7 =& 
\{(\lambda \mid  \mu) \models n \mid m_{4r+2}(\lambda)=0, m_{2r+1}(\lambda)=0, m_{2r+1} (\mu)=0 \mbox { for each } r \mbox { and } \ell(\mu) \equiv 0 \pmod 2 \}.
\end{align*}

Recall from \cref{sec:sym} that $\Lambda_n$ denotes the set of partitions of $n$. We also recall the subsets $\Lambda^1_n$ and $\Lambda^2_n$ from \eqref{lambda1} and \eqref{lambda2} respectively, and the bijection $f: \Lambda^1_n \to \Lambda^2_n$ from \eqref{bij12}. We will think of $f$ as a map from $\Lambda^1_n \to \Lambda_n$. Now define another subset of $\Lambda_n$
\[
\Lambda_n^3=
\{ \lambda \in \Lambda_n \mid \lambda = \langle 1^{2m_1}, 2^{2m_2}, 3^{2m_3}, 4^{2m_4}, \dots \rangle  \},
\]
and another map
$g: \Lambda_n^3 \rightarrow \Lambda_n$ 
as 
\begin{equation}
\label{f2-map-defn}
g( \langle 1^{2m_1}, 2^{2m_2}, 3^{2m_3}, 4^{2m_4}, \dots \rangle ) 
= \langle 2^{m_1},  4^{m_2}, 6^{m_3}, 8^{m_4}, \dots \rangle.
\end{equation} 
We now observe that the map 
$h(\lambda \mid \mu)= (f(\lambda) \mid g(\mu))$ 
is a bijection from $\B_n^3$ to $\B_n^5$. Therefore, we have 
\begin{equation}
\label{eqn:gen-1-D} 	
\sum_{n=0}^{\infty} |\B_n^3|q^n 
= \sum_{n=0}^{\infty} |\B_n^5|q^n 
= \prod_{i=1}^{\infty}\frac{1}{(1-q^{2i-1})(1-q^{4i})(1-q^{2i})}.
\end{equation}
Moreover, the same map $h$ is also a bijection from $\B_n^4$ to $\B_n^6$. This gives us
\begin{equation}
\label{eqn:gen-2-D} 	
\sum_{n=0}^{\infty} |\B_n^4|q^n 
= \sum_{n=0}^{\infty} |\B_n^6|q^n 
 = \prod_{i=1}^{\infty}\frac{1}{(1-q^{4i})(1-q^{2i})}.
\end{equation} 

Finally, the same map $h$ is a bijection from $\B_n^2$ to $\B_n^7$. 
By \cref{lem:type-d-lem-colsum-nonzero-gf-aux}, the generating function for the number of partitions of a positive integer $n$ such that every part is even and the number of parts is also even, is 
\[
\frac{1}{2} \left( 	\prod_{j=1}^{\infty}\frac{1}{1-q^{2j}}+
\prod_{k=1}^{\infty}\frac{1}{1+q^{2k}}
\right).
\] 	
Therefore, we have 
\begin{equation}
\label{eqn:gen-3-D} 	
\sum_{n=0}^{\infty} |\B_n^2|q^n 
= \sum_{n=0}^{\infty} |\B_n^7|q^n
= \frac{1}{2}\left(\prod_{i=1}^{\infty}\frac{1}{1-q^{4i}} \right) 
\left( 	\prod_{j=1}^{\infty}\frac{1}{1-q^{2j}}+
\prod_{k=1}^{\infty}\frac{1}{1+q^{2k}}
\right).
\end{equation}
Using \eqref{eqn:gen-1-D},  
\eqref{eqn:gen-2-D}, and \eqref{eqn:gen-3-D}, we get 
\begin{multline}
\label{eqn:gen-4-D} 
\sum_{n=0}^{\infty} (|\B_n^1|+|\B_n^2|) q^n 
= \sum_{n=0}^{\infty} (|\B_n^3|-|\B_n^4|+|\B_n^5|) q^n \\
= \left(\prod_{i=1}^{\infty}\frac{1}{1-q^{4i}}\right)
\left(	\left(\prod_{j=1}^{\infty}\frac{1}{1-q^{2j}}\right) 
\left( \prod_{k=0}^{\infty}\frac{1}{1-q^{2k+1}}-1 \right) 
+ 
\frac{1}{2}  \left( \prod_{j = 1}^{\infty}\frac{1}{1-q^{2j}}+
\prod_{k = 1}^{\infty}\frac{1}{1+q^{2k}}
\right)  \right).
\end{multline}

By \cref{prop:Conj-class-cplit-in-Dn}, 	
a given bipartition $(\lambda \mid \mu)\models n$ is split 
if and only if $\mu=\emptyset$ and all the parts of $\lambda$ are even. Therefore, each of the bipartitions 
$(\lambda \mid \mu) \models n$ with $\mu=\phi$, $\lambda$ having no odd parts and having even parts with even multiplicity will contribute 
$2$ 
to the number of conjugacy classes with nonzero column sum in $D_n$. 
Therefore, we need to further add the generating function
\[ 
\prod_{i=1}^{\infty}\frac{1}{1-q^{4i}} -1
\] 
to \eqref{eqn:gen-4-D}. This completes the proof.
\end{proof}
		 			 	
For a positive integer $n$, we consider all the set partitions of $[n]$. For any fixed partition of $[n]$, we write all the parts in increasing order and for any part, all the elements other than the lowest and the highest of that part are said to be \emph{transient elements} of that part. For any partition, an element is called \emph{transient} if it is a transient element of some part. For example, let $n=11$ and we consider the following set partition of $\{1,2,\dots,11\} = \{ \{1,2,6\},\{3,9,10,11\}, \{4,5\},\{7\},\{8\} \}.$ 
There are two parts with size more than $2$ and the transient elements of these two parts are $2$ and $9,10$ respectively. Therefore, the transient elements of the partition are $2,9,10$. A set partition of $[n]$ has no transient element if and only if all parts are either singletons or doubletons. 

\begin{theorem}(\cite[Theorem 2]{flajolet-1980})
\label{thm:Flajolet}
Let $\beta_{n_1, n_2, s}$ be the number of partitions having $n_1$ singleton classes, $n_2$ classes of cardinality $\geq 2$ and $s$ non-singleton transient elements. 
Then the generating function 
\[
\beta(u_1, u_2, t, x)= \sum _{n_1, n_2, s \geq 0} \beta_{n_1, n_2, s} u_1^{n_1} u_2^{n_2} t^s x^{s+n_1+2n_2} 
\]
is given by the J-fraction
\[ 
\beta(u_1, u_2, t, x) =  
J(x; (u_1 + nt), ((n+1))u_2) = 
\frac{1}{\ds 1 - u_1 x - \frac{u_2 x^2}{\ds 1 - (u_1+t)x - \frac{2 u_2 x^2}{\ds 1-(u_1+2t)x- \frac{3 u_2 x^2}{\dots}}}}. 
\]
\end{theorem}

Before going to the next proposition, we generalize the sequence 
$o_r(n)$ from \eqref{num odd cycle} and its generating function 
$\mathcal{R}_r(x)$ we defined in \eqref{formula-R} to two variables.
For brevity, we will use the same symbol for both.
Let
\begin{equation}
\label{def-o'}
o_r(m,y) = \sum_{j = 0}^{\lfloor m/2 \rfloor} \sum_{i = 0}^{m-2j} \binom{m}{2j} 
\binom{m-2j}{i} (2j - 1)!! \, r^j y^i.
\end{equation}
Note that $o_r(n, 0) \equiv o_r(n)$.
Similarly, define the J-fraction
\begin{equation}
\label{formula-R'}
\mathcal{R}_r(x, y) = 
J(x; (1+y), ((n+1)r)) = 
\frac{1}{\ds 1 - (1+y) x - \frac{r x^2}{\ds 1 - (1+y) x - \frac{2 r x^2}{\ddots}}}.
\end{equation}
Note that $\mathcal{R}_r(x,0) \equiv \mathcal{R}_r(x)$.

\begin{proposition}
The ordinary generating function of $o_r(m,y)$ is 
\[
\sum_{m = 0}^\infty o_r(m,y) x^m = \mathcal{R}_r(x, y).
\]
\end{proposition}

\begin{proof}
By using \cref{thm:Flajolet}, 
 we observe that $\mathcal{R}_r(x, y) = \beta (1+y, r, 0, x)$. Therefore, it is sufficient to show that 
 \begin{equation}
 \label{eqn:ogf-Rrxy-1}
  \sum_{m = 0}^\infty o_r(m,y) x^m = \sum _{n_1, n_2 \geq 0} \beta_{n_1, n_2, 0} (1+y)^{n_1} r^{n_2}  x^{n_1+2n_2}.
 \end{equation} 
  We compare the coefficient of $x^{m}$ in both sides of \eqref{eqn:ogf-Rrxy-1}. The coefficient of $x^{m}$ in the left hand side is  
\[
\sum_{j = 0}^{m/2} \sum_{i = 0}^{m-2j} \binom{m}{2j} 
\binom{m-2j}{i} (2j - 1)!! \, r^j y^i,
\]
whereas the coefficient of $x^{m}$ in the right hand side of \eqref{eqn:ogf-Rrxy-1} is 
\[
\sum_{n_1, n_2 \geq 0, n_1+2n_2=m} \beta_{n_1, n_2, 0} (1+y)^{n_1} r^{n_2}.
\]
Therefore,
it is sufficient to show that
\begin{equation}
\label{eqn:ogf-Rrxy}
\sum_{j = 0}^{m/2} \sum_{i = 0}^{m-2j} \binom{m}{2j} 
\binom{m-2j}{i} (2j - 1)!! \, r^j y^i   = \sum _{n_1, n_2 \geq 0, n_1+2n_2=m} \beta_{n_1, n_2, 0} (1+y)^{n_1} r^{n_2}.  
\end{equation} 
The coefficient of $y^{i}r^{n_2}$ in the left hand side and the right hand side of \eqref{eqn:ogf-Rrxy} is respectively 
$\binom{m}{2n_2} \binom{m-2n_2}{i} (2n_2 - 1)!!$ and 
$\beta_{n_1, n_2, 0}\binom{n_1}{i}$. 
As $n_1+2n_2=m$, the quantity $\beta_{n_1,n_2,0}$ is the number of partitions of $n_1+2n_2$ into $n_1$ singletons and $n_2$ doubletons and therefore $\beta_{n_1,n_2,0}=\binom{m}{2n_2}(2n_2-1)!!.$ This completes the proof. 
\end{proof}

Let $\Gamma^D_{(\lambda \mid \mu)}$ be the sum of entries of the column indexed by $(\lambda \mid \mu)$ in the character table of $D_n$. 
By \cref{thm:FS}, $\Gamma^D_{(\lambda \mid \mu)}$ is the number of square roots of an element of cycle type $(\lambda \mid \mu)$. 

\begin{proposition}
\label{prop:num-sqrt-Dn}
For elements in $D_n$ whose cycle type contains a single cycle, the following results hold.

\begin{enumerate}

\item 
$\ds \Gamma^D_{((2k)^{2m_{2k}} \mid \emptyset)} = (2m_{2k}-1)!! \, (2k)^{m_{2k}}$,

\item 
$\Gamma^D_{((2k+1)^{m_{2k+1}}) \mid \emptyset)} = o_{2k+1}(m_{2k+1}, 1)$,

\item 
$\Gamma^D_{(\emptyset \mid (k)^{2m_{2k}})} = (2m_{2k}-1)!! \, (k)^{m_{2k}}$.

\end{enumerate}
\end{proposition}

\begin{proof}
By \cref{lem:conjug-class-Dn}, elements in the first two cases belong to $D_n$; for the third case however, $m_{2k}$ must also be even.

By \cref{lem:square-Dn}(2), the square root of a pair of positive cycles of length $2k$ is a positive cycle of length $4k$. We can pair the $2m_{2k}$ cycles in $(2m_{2k}-1)!!$ ways and merge each pair in $2k$ ways. This proves the first part. A similar argument proves the third part.

By \cref{lem:square-Dn}(1) the square root of positive odd cycles can be formed in multiple ways. They can be positive or negative cycles of the same length. If there are two such cycles, their square root can be formed by pairing them into a positive cycle of double the length. In general, we can pair $j$ of these cycles, where $2j \leq m_{2k+1}$, in $(2j-1)!!$ ways and merge each pair in $2k+1$ ways. Among the $m_{2k+1}-2j$ remaining cycles, we can arbitrarily choose $i$ of them, where $i \leq m_{2k+1} - 2j$, to have a square root with a negative cycle. Summing over $i$ and $j$ gives the formula in \eqref{def-o'} with $r = 2k+1$ and $y = 1$, completing the proof.
\end{proof}

Let $s_n^D$ denote the sum of the entries of the character table of $D_n$.
Let $\mathcal{S}^D(x)$ be the ordinary generating function of $(s_n^D)$. Recall the generating function $\mathcal{D}(x)$ for double factorials in \eqref{formula-D}.	Let $\{x_{i}, y_i \mid i \in \mathbb{N} \}$ be a family of commuting indeterminates. 

\begin{theorem}
The number of square roots of an element in $D_n$ with cycle type $(\lambda \mid \mu)$, where they are written in frequency notation as
\[
\lambda = \langle 1^{m_{1}}, 2^{m_{2}}, \dots, n^{m_n}  \rangle
\quad
\text{and}
\quad
\mu = \langle 1^{p_{1}}, 2^{p_{2}}, \dots, n^{p_n}  \rangle,
\]
namely $\Gamma^D_{(\lambda \mid \mu)}$, is the coefficient of $x_1^{m_1} x_2^{m_2} \dots x_n^{m_n} y_1^{p_1} y_2^{p_2} \dots y_n^{p_n}$ after setting all even powers of $z$ to $1$ and odd powers of $z$ to $0$ in
\[
\prod_{k \geq 1} \left( 
\mathcal{D}(4 k x_{2k}^2)
\mathcal{D}(2 k z y_{k}^2) 
\mathcal{R}_{2k-1} (x_{2k-1}, z) \right)  
+ 
\prod_{k \geq 1} \mathcal{D}(4 k x_{2k}^2) 
- 1.
\]
Therefore, the generating function of the sum of the entries in the character table of $D_n$ is 
\[
\mathcal{S}^D(x) = 
\Bigg( \prod_{k \geq 1} \left( \mathcal{D}(4 k x^{4k})
\mathcal{D}(2 k z x^{2k}) \mathcal{R}_{2k-1} (x^{2k-1}, z) \right)  + 
\prod_{k \geq 1} \mathcal{D}(4 k x^{4k}) - 1 \Bigg) \Bigg|_{z_{2k} \to 1, z_{2k-1} \to 0 \forall k}.
\]
\end{theorem}

\begin{proof}
Square elements in $D_n$ are classified in \cref{prop:square-criterion-Dn}.
The idea of this proof is similar to that of \cref{thm:gen-fn-sn}.
However, there are three additional complications. First, we will have to take care that the square roots satisfy \cref{lem:conjug-class-Dn}, i.e. the number of negative cycles must be even. Secondly, we have to ensure that  \cref{prop:square-criterion-Dn}(3) is satisfied, which couples  positive and negative cycles. And thirdly, we have to account for two kinds of conjugacy classes when there are no negative cycles and only even positive cycles according to \cref{cor:conjug-class-Dn}.

The generating function of even positive cycles of length $2k$ is $\mathcal{D}(4kx_{2k}^2)$ using \cref{prop:num-sqrt-Dn}(1) as before, and all square roots have positive cycles. The generating function of negative cycles of length $k$ is $\mathcal{D}(2ky_{k}^2)$ using \cref{prop:num-sqrt-Dn}(3) and all are negative cycles. We now introduce the auxiliary variable $z$ which keeps track of the number of negative cycles in the square root. That is to say, if there are $i$ negative cycles in the square root, we keep an additional factor of $z^i$. Therefore, the latter generating function is $\mathcal{D}(2kzy_{k}^2)$. By the arguments in the proof of \cref{prop:num-sqrt-Dn}(2), the generating function of odd positive cycles of length $2k-1$ is $\mathcal{R}_{2k-1} (x_{2k-1}, z)$. The product of these three terms over all $k$ gives the first product in the generating function. By setting all odd powers of $z$ to $0$ and even powers of $z$ to $1$ ensures that the first condition in the paragraph above is satisfied.

Now suppose we count the number of square roots for an element of cycle type $(\lambda \mid \mu)$. If $\lambda$ has an odd part, the second condition is automatically satisfied. If not, then the factor $\mathcal{R}_{2k-1} (x_{2k-1}, z)$ does not contribute for any value of $k$. We know by \cref{lem:conjug-class-Dn} that $\mu$ has an even number of parts, and for each part, we get a factor of $z$ from the factor $\mathcal{D}(2kzy_{k}^2)$. Since only even powers of $z$ contribute, $\mu$ has length divisible by $4$, proving the second condition.
Finally, suppose $\mu = \emptyset$ and $\lambda$ contains only even parts. As argued above, we only get a contribution from $\prod_k \mathcal{D}(4kx_{2k}^2)$ in that case. But we have two distinct conjugacy classes in that case by \cref{cor:conjug-class-Dn} with the same number of square roots. Therefore, we need another such factor, which explains the second product. This now satisfies the third condition. The last term, $-1$, is present to make sure the constant term is $1$. 

Replacing $x_k$ and $y_k$ by $x^k$ gives the generating function of column sums and completes the proof.
\end{proof}

\subsection{Proof of \cref{property:S}}

Recall the following result (see Corollary 7.9.6 and Theorem 8.3.1 of \cite{musili-1993}]) about irreducible characters of $D_n$.

\begin{theorem}
\label{thm:irr char BD}
Let $\chi^{\alpha \mid \beta}$ denote the irreducible character of $B_n$ corresponding to the bipartition $(\alpha \mid \beta)\models n$. 
Then the following hold:

\begin{enumerate}

\item $\chi^{\alpha \mid \beta} =\chi^{\beta \mid \alpha} \otimes \xi$, where $\xi:B_n \rightarrow \{1,-1\}$ is the
multiplicative character whose kernel is $D_n$.
			
\item If $\chi^{\alpha \mid \beta} \downarrow$ is the restriction of $\chi^{\alpha \mid \beta}$ to $D_n$, then 
$\chi^{\alpha \mid \beta} \downarrow=\chi^{\beta \mid \alpha} \downarrow = \chi^{\alpha \mid\beta}$ if and only if $\alpha \neq \beta$,
and $\chi^{\alpha \mid \alpha} \downarrow = \chi^{\alpha \mid \alpha}_{+} + \chi^{\alpha \mid \alpha}_{-}$. In the latter case, $\chi^{\alpha \mid \alpha}_{\pm}$ are irreducible characters of $D_n$ of equal dimension.

\end{enumerate} 

\end{theorem}

Let $i_n^D$  denote the number of involutions in $D_n$. Recall that $\Gamma^D_{(\lambda \mid \mu)}$ (resp. $\Gamma^B_{(\lambda \mid \mu)}$)
 is the sum of entries of the column indexed by $(\lambda \mid \mu)$ in the character table of $D_n$ (resp. $B_n$). 	
Our next result relates the character sum of $B_n$ with that of $D_n$ when $n$ is odd.

\begin{proposition}
\label{prop:bd col sum}
For an odd positive integer $n$ and $(\lambda \mid \mu)\models n$, we have  
\[
\Gamma^B_{(\lambda \mid \mu)}=
 \begin{cases}
\ds 2 \Gamma^D_{(\lambda \mid \mu)} 
& \text{$\ell(\mu)$  even}, \\
0& \text{$\ell(\mu)$  odd}.
\end{cases}
\]
Therefore, we have $s_n^B  =  2 s_n^D$. 
\end{proposition}

\begin{proof}
Assume that $n$ is odd throughout the proof.
Then there are no irreducible representations of type $(\alpha \mid \alpha)$ in $B_n$. So 
\[
\Irr(B_n)=\{(\alpha \mid \beta), (\beta \mid \alpha) \models n \bigm | \alpha \neq \beta\} \text{ and }  \Irr(D_n)=\{(\alpha \mid \beta)\models n \bigm| \alpha \neq \beta\}.
\] 
By \cref{lem:conjug-class-Dn} an element $w \in B_n$ with cycle type $(\lambda \mid \mu)\models n$ belongs to $D_n$ if and only if $\ell(\mu)$ is even. 
Then, \cref{thm:irr char BD}(1) implies that
 \[
 \chi^{\beta \mid \alpha}((\lambda \mid \mu))=
\begin{cases}
\ds \chi^{\alpha \mid \beta}((\lambda \mid \mu)) & \ell(\mu) \text{ even},\\
\ds -\chi^{\alpha \mid \beta}((\lambda \mid \mu)) & \ell(\mu) \text{ odd}.
\end{cases}
\]
Now, when $(\lambda \mid \mu)\models n$ and $\ell(\mu)$ is even, by \cref{thm:irr char BD}(2) we get
\[
\Gamma^B_{(\lambda \mid \mu)}= 
\sum_{\substack{(\alpha \mid \beta)\models n \\ \alpha \neq \beta}}
\Big( \chi^{\alpha \mid \beta}((\lambda \mid \mu))
 + \chi^{\beta \mid \alpha}((\lambda \mid \mu)) \Big)
 = 
\sum_{\substack{(\alpha \mid \beta)\models n \\ \alpha \neq \beta}}
2\chi^{\alpha \mid \beta}\downarrow((\lambda \mid \mu))
=2 \Gamma^D_{(\lambda \mid \mu)}.
\]
Similarly, if $(\lambda \mid \mu)\models n$ and $\ell(\mu)$ is odd, then 
$\Gamma^B_{(\lambda \mid \mu)}=0$. Note that the latter statement also follows directly from \cref{prop:non-square-B}.
		
The second part follows easily from the first part and the observation that there are no split conjugacy classes in $D_n$ when $n$ is odd.
\end{proof}

When $n$ is even, it seems difficult to estimate the difference of respective column sums 
in the character table of $B_n$ and $D_n$ using direct computation of character values. However, the difference of the first column sum, which is precisely the number of involutions in $B_n$ which do not belong to $D_n$, turns out to have a nice formula. We have not been able to find the following result in the literature. 
It would be interesting to find a combinatorial proof.

\begin{lemma}
	\label{lem:dif-typeb-typed-invol}
	We have
	\[
	2i_n^D-i_n^{B}=
	\begin{cases}
		0 & \text{$n$ odd}, \\
		\ds 2^{n/2}(n-1)!! = \frac{n!}{(n/2)!} & \text{$n$ even}.
	\end{cases}
	\]
	Therefore, for all positive integers $n$,
	$i_n^D \leq i_n^B \leq 2i_n^D$. 
\end{lemma}

\begin{proof}
The result for odd $n$ follows directly from \cref{prop:bd col sum}, but we give a unified algebraic proof for all $n$.
	By \cite[Theorems 2.1 and 3.2]{Chow06}, we have 
	\begin{equation}
		\label{eqn:egf-b-invol}
		\sum_{n\geq0} i_n^B \frac{x^n}{n!}= e^{x^2+2x}, 
	\end{equation}
	and 
	\begin{equation}
		\label{eqn:egf-d-invol}
		\sum_{n\geq0} i_n^D \frac{x^n}{n!}= e^{x^2}\frac{e^{2x}+1}{2}. 
	\end{equation}
	Subtracting \eqref{eqn:egf-d-invol} from \eqref{eqn:egf-b-invol}, we get 
	\begin{equation} 
		\label{eqn:egf-bd-difference-invol}
		\sum_{n\geq0} (i_n^{B}-i_n^{D}) \frac{x^n}{n!}= e^{x^2}\frac{e^{2x}-1}{2}.
	\end{equation}
	Finally, subtracting \eqref{eqn:egf-bd-difference-invol} from \eqref{eqn:egf-d-invol} gives $e^{x^2}$. Thus, there are no odd powers, proving the equation for odd $n$. Extracting the coefficient of $x^{2n}/(2n)!$ proves the result for even $n$.
	
For the second part, it is clear that all the involutions in $D_n$ also belong to $B_n$, proving the first inequality. The second inequality follows from the first part,  completing the proof. 
\end{proof} 
 
Recall the definition of $g_n^B$ from \eqref{def:g function in B}.

\begin{lemma}
\label{lem:bdd-dif-ctsum-firstcolsum}
For positive integers $n$,  
$s_n^D \leq i_n^D + (s_n^B-i_n^B) + g_n^B$.
\end{lemma}

\begin{proof} 
By \cref{prop:Conj-class-cplit-in-Dn}, we have 
\[
s_n^D=
\sum_{(\lambda \mid \mu) \models n \text{ non-split}} \Gamma^D_{(\lambda \mid  \mu)} + 
\sum_{(\lambda \mid \mu) \models n \text{ split}} (\Gamma^D_{(\lambda_+ \mid \emptyset)}  
+ \Gamma^D_{(\lambda_- \mid \emptyset)}).
\]
Since $(\langle 1^n \rangle  \mid \emptyset)$ represents the trivial conjugacy class of $D_n$ and $\Gamma^D_{(\langle 1^n \rangle  \mid \emptyset)} = i_n^D$, we can write
\begin{equation}
	\label{eq:Dcol}
s_n^D=
\left( i_n^D 
+ \sum_{\substack{(\lambda \mid \mu)\models n  \text{ non-split}\\ 
(\lambda \mid \mu)\neq ( \langle 1^n \rangle  \mid \emptyset)}} \Gamma^D_{(\lambda \mid  \mu)} \right) 
+ \sum_{(\lambda \mid \mu) \models n \text{ split}} \Gamma^D_{(\lambda_+ \mid \emptyset)}  
+ \sum_{(\lambda \mid \mu) \models n \text{ split}}\Gamma^D_{(\lambda_- \mid \emptyset)}.
\end{equation}
Since the number of square roots of an element with cycle type $(\lambda \mid \mu)$ in $D_n$ is 
less than or 
equal to the number of square roots of an element with cycle type $(\lambda \mid \mu)$ in $B_n$,
we have $\Gamma^D_{(\lambda \mid \mu)} \leq \Gamma^B_{(\lambda \mid \mu)}$, 
the $(\lambda \mid \mu)$'th column sum 
in the character table of $B_n$. Applying this upper bound for the second and third summand (of the right hand side) of \ref{eq:Dcol}, we get 
\begin{equation}
\label{eqn:asymptotics-type-d-first-b}
s_n^D  \leq  i_n^D
+ \sum_{\substack{(\lambda \mid \mu)\models n  \text{ and }\\
 (\lambda \mid \mu)\neq (\langle 1^n \rangle \mid \emptyset)}} \Gamma^B_{(\lambda \mid  \mu)}   
+ \sum_{(\lambda \mid \mu) \models n \text{ is split}}\Gamma^D_{(\lambda_- \mid \emptyset)}.
\end{equation}

The second sum of the right hand side of \eqref{eqn:asymptotics-type-d-first-b} is clearly equal to $s_n^B-i_n^B$.
By \cref{prop:Conj-class-cplit-in-Dn}, if $(\lambda \mid \mu)\models n$ is a split bipartition,  then $\mu=\emptyset$ and all parts of $\lambda$ are even. In particular, a split bipartition $(\lambda \mid \mu)$ has the property $m_1(\lambda)=0$.  Thus, the last sum of \eqref{eqn:asymptotics-type-d-first-b} is at most $g_n^B$ .
Therefore, we have the desired inequality. 
\end{proof} 
  
\begin{theorem}
\label{thm:bothsidebounds-d} 
\cref{property:S} holds for all demihyperoctahedral groups. 
\end{theorem}

\begin{proof}
By \cref{prop:bd col sum} and \cref{lem:dif-typeb-typed-invol}, \cref{property:S} holds for $D_n$ when $n$ is odd. However we will not use this fact for our proof strategy.

First we consider the case when $n \geq 16$. Using \eqref{eqn:bnbounds} in \cref{lem:bdd-dif-ctsum-firstcolsum} we have,
\begin{equation}
\label{eqn:type-d-finalthm-1} 
s_n^D-i_n^D 
\leq 
2 (i_{n-1}^B+i_{n-2}^B)+g_n^B.
\end{equation} 
Setting $k=1$ in \cref{lem:typea-asymp-main-b} (since $n \geq 5$), we have 
\begin{equation}
\label{eqn:type-d-finalthm-2} 
s_n^D-i_n^D 
\leq 
2 i_{n-1}^B+2 i_{n-2}^B+\frac{i_n^B}{n-1}.
\end{equation}
Now apply the first inequality of \cref{lem:type-b-invol-ratio} to the second term of the right hand side of \eqref{eqn:type-d-finalthm-2} to get
\begin{equation}
\label{eqn:type-d-finalthm-3} 
s_n^D-i_n^D 
\leq 
2 i_{n-1}^B+ \frac{2i_{n-1}^B}{\sqrt{2n-2}}+\frac{i_n^B}{n-1} \leq \left(2+\frac{2}{\sqrt{30}} \right) i_{n-1}^B+\frac{i_n^B}{15},
\end{equation} 
where we have used $n \geq 16$ for the last inequality. 
For $n\geq 16$, again using the first inequality of \cref{lem:type-b-invol-ratio}, we have $i_{n-1}^B \leq i_n^B/\sqrt{2n} \leq i_n^B/\sqrt {32}$.  Using this fact along with the inequality $i_n^B \leq 2i_n^D$ from \cref{lem:dif-typeb-typed-invol} in \eqref{eqn:type-d-finalthm-3}, we get
\begin{equation*}
\label{eqn:type-d-finalthm-5} 
s_n^D-i_n^D
\leq 
2 \frac{(2\sqrt{30}+2)}{\sqrt{30 \cdot 32}} i_n^D+ \frac{2}{15} i_{n}^D \leq  i_n^D,
\end{equation*}
where the last inequality follows because
\[
2 \frac{(2\sqrt{30}+2)}{\sqrt{30 \cdot 32}} + \frac{2}{15} 
\approx 0.9695 \leq 1. 
\] 
This completes the proof for $n \geq 16$. 
For $n \leq 15$, the validity of the statement holds by the direct computation shown in \cref{sec:small demi}. 
Thus, we have proved the result for all positive integers $n$.
\end{proof}

\begin{corollary}
\label{cor:asymptotics-typed-invol}
For positive integers $n$, we have
\[
s_n^D \sim  \frac{e^{\sqrt{2n}}}{ 2\sqrt{2e} } \left(\frac{2n}{e} \right)^{n/2}.
\]
\end{corollary} 

\begin{proof}
This follows immediately by using \cref{cor:asymp-B} and \cref{lem:dif-typeb-typed-invol}. 
\end{proof}

\section{Generalized symmetric groups}
\label{sec:wreath}
We follow \cite[Section 2]{mishra-paul-singla-2023} for the notational background used in this section.
For nonnegative integers $r,n$, let $\Z_r=\{\overline{0},\overline{1},\ldots,\overline{r-1}\}$ be the additive cyclic group of order $r$, where we use bars to distinguish these elements from those in the symmetric group. Then define the \emph{generalized symmetric group} 
\[
\Grn = \Z_r \wr S_n:=\{(z_1,\dots,z_n;\sigma) \mid z_i\in \Z_r, \sigma \in S_n\}.
\] 
We remark here that the complex reflection group $G(r,q,n)$ is an index $q$ subgroup of $\Grn$; see \cref{sec:Grn-abs-sq} for details.
If $\pi=(z_1, z_2, \ldots, z_n; \sigma)$ and $\pi'=(z_1', z_2', \ldots, z_n'; \sigma')$, then their product is given by 
\begin{equation}
\label{prod-Grn}
\pi \, \pi'=(z_1+z_{\sigma^{-1}(1)}',\ldots,z_n+z_{\sigma^{-1}(n)}';\sigma \sigma'),
\end{equation}
where $\sigma \sigma'$ is the standard product of permutations in $S_n$.

The group $\Grn$  can also be realized as a subgroup of the symmetric group $S_{rn}$.
In this interpretation $\Grn$ consists of all permutations $\pi$ of the set $\{\overline{k}+ i \mid 0\leq k \leq r-1, 1\leq i\leq n \} $  satisfying $\pi(\overline{k}+ i)=\overline{k} + \pi(i)$ for all allowed $k$ and $i$. 
For convenience, we identify the letters $\overline{0}+ i$ with $i$ for $1\leq i \leq n$. Given a permutation $\pi \in \Grn$, its values at  $1,\ldots,n$ determine $\pi$ uniquely. 

The two definitions above are identified using the bijective map $\phi$ defined on the window $[1, \dots, n]$ by 
\[
\phi((z_1,\ldots,z_n;\sigma))=
\begin{bmatrix}
1 & 2 & \cdots & n\\
z_{\sigma (1)} +\sigma(1) & z_{\sigma(2)} +\sigma(2) & \cdots & z_{\sigma(n)} +\sigma(n)
\end{bmatrix}.
\]
This map satisfies $\phi(\pi \, \pi')=\phi(\pi)\circ\phi(\pi')$, where $\circ$ is the usual composition of permutations in $S_{rn}$.

Let $\pi=(z_1, \ldots, z_n; \sigma)\in \Grn$ and  $(u_1), \ldots ,(u_t)$ be the cycles of $\sigma$. 
Let $(u_i)=(u_{i,1},  \ldots, u_{i,{\ell_i}})$ where $\ell_i$ is the length of the cycle $(u_i)$.
Define the \emph{color} of a cycle $(u_i)$ as $z(u_i) := z_{u_{i, 1}}+ z_{u_{i, 2}} +\ldots +z_{u_{i, \ell_i}} \in \mathbb{Z}_r$. 
For $j\in\{0,\ldots, r-1\}$, let $\lambda^j$ be the partition formed by the lengths of color $j$ cycles of $\sigma$. Note that $\underset{j}{\sum} |\lambda^j|=n$.
 The $r$-tuple of partitions $\bla = (\lambda^0 \mid \lambda^1 \mid \ldots \mid \lambda^{r-1})$ is called the \emph{cycle type} of $\pi$. We refer to such an $r$-tuple of partitions as an \emph{r-partite partition} of size $n$, denoted $\lambda \models_r n$.

For example, the cycle type of the element
\[
(\overline{2},\overline{1},\overline{1},\overline{1},\overline{0},\overline{2}; \, (123)(45)(6))\in G(3,1,6)
\]
is $(\emptyset \mid (3,2) \mid (1))$.

The following theorem asserts that the conjugacy classes of $\Grn$ are indexed by $r$-partite partitions of $n$.

\begin{theorem}{\cite[p. 170]{macd}}
Two elements $\pi_1$ and $\pi_2$ in $\Grn$ are conjugate if and only if their corresponding cycle types are equal.
\end{theorem}

\begin{lemma}
\label{lem:squares Grn}
Let $\pi=(z_1,\ldots,z_d;\sigma)\in G(r,1,d)$ with a single cycle of color $t$.
\begin{enumerate}
\item When $d$ is even, $\pi^2$ is a product of two color $t$ cycles each of length $d/2$.
\item When $d$ is odd, $\pi^2$ is a single cycle of length $d$ and color $2t$.
\end{enumerate}
\end{lemma}	

\begin{proof}
Let $\pi=(z_1,\ldots,z_d;\sigma)\in G(r,1,d)$ where $\sigma=(d,d-1,\ldots,1)$. Then, $\pi^2=(z_1+z_2,\ldots,z_d+z_{1};\sigma^2)$.
When $d$ is even, $\sigma^2=(d,d-2,\ldots,4,2)(d-1,d-3,\ldots,3,1)$. Note that the color of each cycle of $\sigma^2$ is $ z_1+\cdots+z_d$. Thus the first statement follows.
When $d$ is odd, $\sigma^2=(d,d-2,\ldots,3,1,d-1,d-3,\ldots, 2)$ and the color of $\pi^2$ is $(z_1+z_2)+\cdots + (z_d+z_1)=2(z_1+\cdots+z_d)$. This concludes the proof.
\end{proof}

\begin{proposition}
\label{thm:squares Grn}
An $r$-partite  partition $\bla = (\lambda^0 \mid \lambda^1 \mid \ldots \mid \lambda^{r-1})$ is the cycle-type of a square element of $\Grn$ if and only if  the following conditions hold:
\begin{enumerate}
\item each even part in each partition $\lambda^t$ has even multiplicity, and
\item when $r$ is even, odd parts in $\lambda^1, \lambda^3, \ldots, \lambda^{r-1}$ have even multiplicity. 
\end{enumerate}
\end{proposition}

\begin{proof}
Let $\pi$ have cycle type $\bla = (\lambda^0 \mid \lambda^1 \mid \ldots \mid \lambda^{r-1})$ and $w$ be a square root of $\pi$.
The first part follows directly from \cref{lem:squares Grn}(1). When $r$ is even, a single odd length cycle of $\pi$ with color $2k-1$, where $k=1,2,\ldots,r/2$, cannot be obtained by squaring any single cycle (of any color) of $w$. 
This is because, if a cycle of color $t$ of $w$ squares to a cycle of color $2k-1$ of $\pi$, then $2t=2k-1$ by \cref{lem:squares Grn}(2). But this equation has no solution in $\Z_r$ when $r$ is even. Thus, all odd parts in $\lambda^1,\lambda^3,\ldots,\lambda^{r-1}$ must come from squaring even length cycles of the same color, enforcing the even multiplicity condition. 

The converse part follows by applying \cref{lem:squares Grn}. A detailed construction of such a square root is demonstrated in the proof of \cref{prop:num-sqrt}.
\end{proof}

\subsection{Generating function for number of square roots}
\label{sec:Grn-sqrt}

We now count the number of square roots for elements satisfying \cref{thm:squares Grn}. 
We say $\pi=(z_1,\ldots,z_d;\sigma) \in G(r, 1, d)$ is a single cycle of color $t$ if $\sigma \in S_d$ is a $d$-cycle with color $z(\sigma)=t$.

\begin{proposition}
\label{prop:num-sqrt}

\begin{enumerate}

\item  Let $\pi \in G(r, 1, d)$ be a single cycle of color $t$ with $d$ odd. Then the number of square roots of $\pi$ is
\[
\begin{cases}
1 & \text{$r$ odd}, \\
2 & \text{$r$ even and $t$ even}, \\
0 & \text{$r$ even and $t$ odd}.
\end{cases}
\]

\item Let $\pi \in G(r, 1, 2d)$ be a product of two disjoint cycles, both having color $t$ and length $d$. Then the number of square roots of $\pi$ is $dr$.

\end{enumerate}
\end{proposition}

\begin{proof}
 Without loss of generality, let $\pi=(t,0,\ldots,0;\sigma=(d,d-2,\dots,1,d-1,d-3,\ldots,2))$, which has cycle type $\lambda^t=(d)$, with all other $\lambda^i$'s being empty. 
  When $d$ is odd, the cycle $\sigma$ has the unique square root $(d,d-1,\ldots,1)$ in $S_d$. Let $(z_1,\ldots,z_d; (d,d-1,\ldots,1))$ be a square root of $\pi$. This gives the following system of equations:
\[ 
 z_1+z_2=t, z_2 +z_3=0, \ldots, z_{d-1} + z_d=0, z_{d} + z_1=0.
 \]
Solving this system, we get $2z_2=t$.
When $r$ is odd, $2z_2=t$ has a unique solution in $\Z_r$, and thus $\pi$ has unique square root.
 When $r$ is even, the equation $2z_2=t$ in $\Z_r$ has no solution if $t$ is odd and has exactly two solutions if $t$ is even. This proves the first part.
 
Finally, let $\pi=(t,0,\ldots,0,t ; \sigma) \in G(r,1,2d)$ where $\sigma = (2d,2d-2,\ldots,4,2)(2d-1,2d-3,\ldots,3,1)$. Note that $\sigma$ has $d$ square roots in $S_{2d}$, one of which is $(2d, 2d-1, \dots, 1)$. 
Let $(z_1,\ldots,z_{2d};(2d,2d-1,\ldots, 1))$ be a square root of $\pi$. Then we get the following system of equations:
\[
z_1+ z_2=t, z_2+z_3=0,\ldots,
z_{2d-1}+z_{2d}=0, z_{2d}+z_{1}=t.
\]
Observe that the above system has $r$ solutions as every choice of $z_1$ is valid and fixes all other values. For each choice of square root of $\sigma$, we have $r$ square roots of $\pi$. This concludes the proof.
 \end{proof}

\begin{proposition}
\label{cor:num-sqrt}
Given a positive integer $r$ and $t$ such that $0\leq t \leq r-1$, the number of square roots of elements whose cycles all have the same length $\ell$ and color $t$ is as follows.

\begin{enumerate}
\item When $\ell = 2k$ and there are $2m$ cycles, this is given by
\begin{equation}
\label{num-sqrt1}
(2m-1)!! (2k)^{m} r^{m}.
\end{equation}

\item When $\ell = 2k+1$, $r$ is odd and there are $m$ cycles, this is given by
\begin{equation}
\label{num-sqrt2}
\sum_{j=0}^{\lfloor \frac{m}{2}\rfloor} \binom{m}{2j} (2j-1)!! (2k+1)^j r^j.
\end{equation}

\item When $\ell = 2k+1$, both $r$ and $t$ are even and there are $m$ cycles, this is given by
\begin{equation}
\label{num-sqrt3}
\sum_{j=0}^{\lfloor \frac{m}{2}\rfloor} \binom{m}{2j} (2j-1)!! (2k+1)^j r^j 2^{m-2j}.
\end{equation}

\item When $\ell = 2k+1$, $r$ is even, $t$ is odd and there are $m$ cycles, this is given by
\begin{equation}
\label{num-sqrt4}
(2m-1)!! (2k+1)^{m} r^{m}.
\end{equation}
\end{enumerate}

\end{proposition}

\begin{proof}
We want to consider all elements $\pi$ whose square is an element $\pi'$ whose cycle type has only cycles of length $\ell$ with color $t$.
Let the cycle type of $\pi$ be $\bla = (\lambda^0 \mid \dots \mid \lambda^{r-1})$.

First suppose $\ell = 2k$, and there are $2m$ cycles.
By \cref{lem:squares Grn}(1), the square roots of such element will necessarily arise from elements whose cycle type $\lambda^t$ contains  $(4k)^{m}$ by combining cycles in pairs and choosing the color in one of $r$ ways as in \cref{prop:num-sqrt}(3). The number of ways of doing this is given in \eqref{num-sqrt1}.

Now suppose $\ell = 2k + 1$. 
Then this calculation depends on whether $r$ is odd or even, and in the latter case, whether $t$ is odd or even. The simplest case is when $r$ is odd, and if so, suppose there are $m$ cycles.
By \cref{lem:squares Grn}(1) and (2), the square roots of this element can either contain $(2k+1)$ cycles in the partition $\lambda^{t/2}$, (where $t/2$ is well-defined in $\Z_r$ because $r$ is odd) or $(4k+2)$ cycles in the partition $\lambda^t$ (by pairing cycles and coloring each of them in $r$ ways). 
Any combination of singletons and pairs gives a valid way of creating a square root.
The number of ways of doing this is given in \eqref{num-sqrt2}.

When $r$ and $t$ are both even, suppose again that there are $m$ cycles.
The calculation is very similar to the above case, except that there are two choices for $t/2$ in $\Z_r$ as shown in \cref{prop:num-sqrt}(2). Therefore, we get an extra power of $2$ proving \eqref{num-sqrt3}.

Lastly, when $r$ is even, $t$ is odd and 
there are $2m$ cycles,
the situation is much simpler. In this case, 
$t/2$ does not exist in $\Z_r$ and therefore the square root cannot contain $(2k+1)$ cycles in $\lambda^{t/2}$ as shown in \cref{prop:num-sqrt}(2). Thus, the cycles will have to pair up just as in the first case of even cycles and the square root has to contain $(4k+2)^{m}$ in $\lambda^t$. The number of ways of pairing these cycles is thus \eqref{num-sqrt4}, completing the proof.
\end{proof}

Recall the generating functions $\mathcal{D}(x)$ from \eqref{formula-D} and $\mathcal{R}_r(x)$ from \eqref{formula-R}. Let $x, \{x_{t,j} \mid 0 \leq t \leq r-1, j \in \mathbb{N} \}$ be a family of commuting indeterminates.

\begin{theorem}
\label{thm:gen-fun-sqrt}
Write each partition in the $r$-partite partition $\bla = (\lambda^0 \mid \lambda^1 \mid \ldots \mid \lambda^{r-1})$ of size $n$
 in frequency notation, with
$\lambda^t = \langle 1^{m_{t,1}}, 2^{m_{t,2}}, \dots, n^{m_{t,n}}  \rangle$.
Then the number of square roots of an element in $\Grn$ with cycle type $\bla$ is the coefficient of 
\[
(x_{0, 1}^{m_{0, 1}} \dots x_{0, n}^{m_{0, n}})
\dots
(x_{r-1, 1}^{m_{r-1, 1}} \dots x_{r-1, n}^{m_{r-1, n}})
\]
in
\[
\begin{cases}
\ds \prod_{k \geq 1} \prod_{t = 0}^{r-1}   \left( \ds 
\mathcal{D}	(2kr x_{t,2k}^{2}) \,
\mathcal{R}_{r(2k-1)} (x_{t,2k-1}) \right)
& r \text{ odd}, \\
\ds \prod_{k \geq 1} \left( \ds 
\prod_{t = 0}^{r-1} \mathcal{D}(2kr x_{t,2k}^2) \,
\prod_{\substack{t = 0 \\ t \, \text{odd} }}^{r-1} \mathcal{D}((2k-1)r x_{t,2k-1}^{2})  \,
\prod_{\substack{t = 0 \\ t \, \text{even}}}^{r-1} \mathcal{R}_{\frac{r(2k-1)}{4}} (2x_{t,2k-1}) \right)
& r \text{ even}.
\end{cases}
\]
Therefore, the generating function (in $n$) for the sum of the number of square roots of all the conjugacy class representatives in $\Grn$ is
\[
\begin{cases}
\ds \prod_{k \geq 1} \left( \ds \mathcal{D}(2kr x^{4k})^r \,
\mathcal{R}_{r(2k-1)} (x^{2k-1})^r \right)
& r \text{ odd}, \\
\ds \prod_{k \geq 1} \left( \ds \mathcal{D}(2kr x^{4k})^r \,
\mathcal{D}((2k-1)r x^{4k-2}) ^{r/2} \,
\mathcal{R}_{\frac{r(2k-1)}{4}} (2x^{2k-1}) ^{r/2} \right)
& r \text{ even}.
\end{cases}
\]
\end{theorem}

\begin{proof}
Suppose $\bla = (\lambda^0 \mid \lambda^1 \mid \ldots \mid \lambda^{r-1}) \models_r n$ and $\pi$ is an element with cycle type $\bla$. Then \cref{thm:squares Grn} gives the conditions under which $\pi$ has a square root.  Thus, write each of these partitions in frequency notation as
\begin{equation}
\lambda^t = 
\begin{cases}
\langle 1^{m_{t,1}}, 2^{2m_{t,2}}, 3^{m_{t,3}}, 4^{2m_{t,4}}, \dots \rangle 
& \text{$r$ is odd, or both $r$ and $t$ are even}, \\
\langle 1^{2m_{t,1}}, 2^{2m_{t,2}}, 3^{2m_{t,3}}, 4^{2m_{t,4}}, \dots \rangle 
& \text{$r$ is even and $t$ is odd}.
\end{cases}
\end{equation}
If only one of the $m_{t,j}$'s is non-zero, the number of square roots is given by \cref{cor:num-sqrt}.

We now have to argue that the square roots for the cycles in $\pi$ of colour $t$ can be formed independently of other cycles of different lengths or colors. When the cycles are paired up to form an even cycle, they remain of color $t$ by \cref{lem:squares Grn}(1). But if a cycle stays of the same size under the square root operation, and that can only happen if the cycle is odd, then it changes color (unless $t = 0$) to $t/2$ by \cref{lem:squares Grn}(2). In either case, the nature of this cycle in the square root is unaffected by the other cycles. Therefore, the number of square roots for $\pi$ can be computed by taking a product over all colors and all cycle lengths independently.

By the previous paragraph, the generating function of the number of square roots is the product of those for different cycles of different colors.
Even cycles $(2k)$ of color $t$ with multiplicity $2m_{t,2k}$ contribute $x_{t,2k}^{2 m_{t,2k}}$ to the generating function and we get, by summing over $m_{t,2k}$ and using \cref{cor:num-sqrt}(1) and \eqref{formula-D}, a factor of $\mathcal{D}	(2kr x_{t,2k}^{2})$ for each $t$. This does not depend on the parity of $r$.

When $r$ is odd, odd cycles $(2k-1)$ of color $t$ with multiplicity $m_{t,2k-1}$ contribute $x_{t, 2k-1}^{m_{t,2k-1}}$ to the generating function. By summing over $m_{t,2k-1}$ and using \cref{cor:num-sqrt}(2) and \eqref{formula-R}, we get a factor of $\mathcal{R}_{r(2k-1)}(x_{t, 2k-1})$ for each $t$. This proves the formula when $r$ is odd.

Suppose now that both $r$ and the color $t$ are even. Then odd cycles $(2k-1)$ of color $t$ with multiplicity $m_{t,2k-1}$ contribute $x_{t, 2k-1}^{m_{t,2k-1}}$ to the generating function. From the formula in \cref{cor:num-sqrt}(3), the part of the summand with exponent $j$ is $r(2k-1)/4$ and we have an extra factor of $2^{m_{t,2k-1}}$. Therefore, summing over $m_{t,2k-1}$ and using \eqref{formula-R} will give us $\mathcal{R}_{r(2k-1)/4}(2x_{t, 2k-1})$ for every even $t$.

The last case is when $r$ is even and the color $t$ is odd. 
Then odd cycles $(2k-1)$ of color $t$ with multiplicity $2m_{t,2k-1}$ contribute $x_{t, 2k-1}^{2 m_{t,2k-1}}$ to the generating function. This calculation is very similar to the case of even cycles. Comparing \cref{cor:num-sqrt}(1) and \cref{cor:num-sqrt}(4), we obtain a factor of $\mathcal{D}((2k-1) r x_{t, 2k-1}^{2})$ for each even $t$. This explains all the factors in the case when $r$ is even.

To obtain the generating function, replace each $x_{t,k}$ by $x^k$, completing the proof.
\end{proof}

\subsection{Generating function for character table sums}
\label{sec:Grn-abs-sq}
In contrast with the case of $S_n$, the square root function does not give column sums for character table of $G(r,1,3)$, $r>2$ as the group has non-real irreducible characters~\cite{adin-postnikov-roichman-2010}. Given $\pi=(z_1, z_2, \ldots, z_n; \sigma)\in \Grn$, define the 
\emph{bar operation} as 
\[
\overline{\pi}:=(-z_1,\ldots,-z_n;\sigma).	\]
An element $g\in \Grn$ is said to have an \emph{absolute square root} if there exists $\pi\in\Grn$ such that $\pi \, \overline{\pi}=g$. 
We will call the element $\pi \, \overline{\pi}$ the \emph{absolute square} of $\pi$. Note that $\overline{\pi} \, \pi = \overline{(\pi \, \overline{\pi})}$.
The next result describes columns sum for $\Grn$ in terms of absolute square roots.
	
\begin{theorem}{\cite[Theorem 3.4]{adin-postnikov-roichman-2010}}
\label{thm:APR}
Let $\{\chi_{\bla} \mid \bla \models_r n\}$ be the set of irreducible characters of $\Grn$. Then
\begin{equation*}
\sum_{\bla \models_r n} \chi_{\bla}(g)=|\{\pi \in \Grn \mid \pi \, \overline{\pi}= g\}| \quad \forall g \in \Grn.
\end{equation*}
\end{theorem}

Let $r,q,n$ are positive integers such that $q \mid r$. Then the \emph{complex reflection group} $G(r,q,n)$ is defined by
\[
G(r,q,n):=\left\{ \pi=(z_1,\ldots,z_n;\sigma) \in \Grn \mid \sum_{i=1}^{n} z_i \equiv 0 \text{ mod } q \right\}.
\]
The group $G(r,q,n)$ is a normal subgroup of $\Grn$ with index $q$. Special cases of $G(r,q,n)$ where $q \geq 2$ include the demioctahedral groups $D_n=G(2,2,n)$ and the dihedral groups $\dih(n)=G(n,n,2)$.

\begin{remark}
	\label{rem: pos col sum}
 By \cref{thm:APR}, all the column sums are nonnegative integers for $\Grn$. The same is not true for all $G(r,q,n)$. For example, there are six columns in the character table of $G(9,3,3)$ with column sum $-6$. 
Caselli~\cite[Section~4]{caselli-2010} building on work of Bump--Ginzburg~\cite{bump-ginzburg-2004} proved that
\begin{equation}
	\label{eq:abs sq}
\sum_{V \in \Irr(G(r,q,n))} \chi_{V}(g) = 
|\{\pi \in G(r,q,n) \mid \pi \, \overline{\pi}= g\}| \quad \forall g \in G(r,q,n)
\end{equation}
if and only if $\gcd(q,n)\leq 2$.
This result completely classifies complex reflection groups whose column sums are given by absolute square roots.
\end{remark}

\begin{theorem}
	\label{thm: conj Grn}
Suppose $r,q,n$ are positive integers such that $q \mid r$ and $\gcd(q,n)\leq 2$. Then \cref{conj:lower bound} holds for $G(r,q,n)$. 
\end{theorem}	

\begin{proof}
By \cref{rem: pos col sum}, the column sums are nonnegative integers for $G=G(r,q,n)$. Thus, $s(G) \geq \Gamma_e(G)$. For the remaining part, assume that $s(G)=\Gamma_e(G)$. We will show that $G$ must be abelian.
 
Since $s(G)=\Gamma_e(G)$, by \eqref{eq:abs sq} we have  $\pi \, \overline{\pi}=e$ for all $\pi=(z_1,\ldots,z_n;\sigma) \in G$, where $e=(0,\ldots,0; 12 \dots n)$ is the identity element in $G$ and $12 \dots n$ is the identity in $S_n$ in one-line notation. 
By \eqref{prod-Grn}, we have $\sigma^2= 12 \dots n$  for all $\sigma \in S_n$.
Thus, $S_n$ is abelian and this happens only when $n\leq 2$.

If $n=1$, then $G$ is an index $q$ subgroup of the cyclic group $\Z_r$ and so $G$ is abelian. 
Let $n=2$. If $r=1$, then $G=G(1,1,2)$ is the abelian group $S_2$. 
If $r = 2$, $G$ is either $G(2,1,2) = \dih(4)$ or $G(2,2,2) = \dih(2)$.
For the former nonabelian group, the inequality can be verified directly. The latter is an abelian group.

Finally, consider the case of $G=G(r,q,2)$ with $r>2$.
Observe that the element $\pi=(r-1,1;(1,2))\in G$ is an absolute square root of $g=(r-2,2;12)$, that is $\pi \, \overline{\pi}=g$. Since $r\neq 2$, $g$ is not the identity element. Now, by \eqref{eq:abs sq} we conclude that the column sum corresponding to $g$ is positive and thus $s(G)> \Gamma_e(G)$. 
\end{proof}

By analyzing absolute square roots, we will provide generating functions for number of columns with zero sums and the character table sums of $\Grn$.
	
\begin{lemma}
\label{lem:abs-square-color-criterion} 
\begin{enumerate}

\item The absolute square of a cycle of odd length $d$ (of any color) is a cycle of the same length of color $0$.

\item The absolute square of a cycle of even length $d$ (of any color) is a product of two cycles, each of length $d/2$, such that sum of their colors is zero.

\end{enumerate}
\end{lemma}
	
\begin{proof}
Let $\pi=(z_1,\ldots,z_d;\sigma)\in G(r,1,d)$ where $\sigma=(d,d-1,\ldots,1)$. Then, its absolute square is  $\pi \, \overline{\pi}=(z_1-z_2,\ldots,z_{d-1}-z_d,z_d-z_{1};\sigma^2)$.
		
When $d$ is odd, $\sigma^2=(d,d-2,\ldots,3,1,d-1,d-3,\ldots, 2)$ and the color of $\sigma^2$ is zero as $(z_1-z_2)+\cdots + (z_{d-1}-z_d) + (z_d-z_1)=0$.		
Similarly, when $d$ is even, $\sigma^2=(d,d-2,\ldots,4,2)(d-1,d-3,\ldots,3,1)$. In this case notice that the sum of the colors of both cycles is zero.
\end{proof}
        
The following result characterizes columns of the character table of $\Grn$ having non-zero sums.

\begin{proposition}
\label{prop:abs-sqr-gen-fun}
An $r$-partite  partition $\bla = (\lambda^0 \mid \lambda^1 \mid \ldots \mid \lambda^{r-1}) \models_r n$ is the cycle type of an absolute square in $\Grn$ if and only if the following hold:
\begin{enumerate}
\item each even part in $\lambda^0$ has even multiplicity,
\item  $\lambda^t=\lambda^{r-t}$ for all $1 \leq t < r/2$, and
\item each part in $\lambda^{r/2}$ has even multiplicity when $r$ is even.
\end{enumerate}
\end{proposition}

\begin{proof}
Let $\pi$ be an absolute square having cycle type $\bla = (\lambda^0 \mid \lambda^1 \mid \ldots \mid \lambda^{r-1})$. By \cref{lem:abs-square-color-criterion}, all even parts in $\lambda^0$ must come in pairs, enforcing the even multiplicity condition. For the same reason, all parts in $\lambda^{r/2}$ must come in pairs when $r$ is even.
This proves the first and the third part.
The second part follows directly 
from \cref{lem:abs-square-color-criterion}(2). 
For the converse, we can use \cref{lem:abs-square-color-criterion} to construct an explicit absolute square root of a given element $\pi$ with cycle type $\bla$ satisfying conditions (1),(2) and (3). This construction is demonstrated in the proof of \cref{lem:abs-square-roots-no}.
\end{proof}

The following result extends Bessenrodt--Olsson's result \cref{prop:bes-ols} from $S_n$ to $\Grn$.	   

\begin{theorem}
	\label{thm:abs-sqr-gen-fun}
The generating function for the number of conjugacy classes of $\Grn$ with nonzero column sum is
\[
\begin{cases}
\ds  \prod_{i=1}^{\infty}\frac{1}{(1-q^{2i-1})(1-q^{4i})(1-q^i)^{(r-1)/2} }
& r \text { odd, }\\
\ds 
\prod_{i=1}^{\infty}\frac{1}{(1-q^{2i-1})(1-q^{4i})(1-q^{2i})(1-q^i)^{(r-2)/2}}
& r \text{ even. }
\end{cases}
\]

\end{theorem}
 
\begin{proof}
We first consider the case when $r$ is odd. Let 
\[ 
U_1(n)= 
\{ \bla = (\lambda^0 \mid \lambda^1 \mid  \dots \mid \lambda^{r-1}) \models_r n \mid  m_{2j}(\lambda^0)  \mbox { is even for each } j, 
\lambda^i= \lambda^{r-i} \mbox{ for each } 
1 \leq i \leq (r-1)/2\}.
\]
By \cref{prop:abs-sqr-gen-fun}, an element
$\pi \in \Grn$
with cycle type $\bla=( \lambda^0 \mid \lambda^1 \mid  \dots \mid \lambda^{r-1})$ is the absolute square of some element 
if and only if 
$\bla \in U_1(n)$. To determine the generating function for $U_1(n)$, we define the set
\[ 
U_2(n)= 
\{ \bla = (\lambda^0 \mid \lambda^1 \mid  \dots \mid \lambda^{r-1}) \models_r n \mid  m_{4j+2}(\lambda^0) =0 \mbox{ for each } j, 
\lambda^i= \lambda^{r-i} \mbox{ for each }
1 \leq i \leq (r-1)/2\}.
\]
Recall the map $f$ defined in \eqref{bij12}. Using that, we define the map $f^{*}:U_1(n) \mapsto U_2(n)$ as
\[f^{*}(\lambda^0 \mid \lambda^1 \mid  \dots \mid \lambda^{r-1})
= (f(\lambda^0) \mid \lambda^1 \mid  \dots \mid \lambda^{r-1}).\]
By \cref{cor:bes-ols}, the map $f^*$ maps 
$U_1(n)$ to $U_2(n)$ bijectively. Therefore, the generating function for 
the number of partitions in $\lambda^0$ is 
\[
 \prod_{i=1}^{\infty}\frac{1}{(1-q^{2i-1})(1-q^{4i})}
\]
Moreover, as there is no restriction on the number of 
parts of $\lambda^i$ for $1 \leq i \leq (r-1)/2$, the generating function for 
the number of partitions for each  such $\lambda^i$ is 
\begin{equation}
\label{partn-gf}
\prod_{i=1}^{\infty}\frac{1}{1-q^{i}}.
\end{equation}
As $\lambda^i= \lambda^{r-i}$, we have
\[
\sum_{n=0}^{\infty} |U_1(n)| q^n = 
\sum_{n=0}^{\infty} |U_2(n)| q^n =
\ds  \prod_{i=1}^{\infty}\frac{1}{(1-q^{2i-1})(1-q^{4i})(1-q^i)^{(r-1)/2} }.
\] 
This completes the proof for odd $r$. 

When $r$ is even, we get an extra factor because of $\lambda^{r/2}$. 
From \cref{prop:abs-sqr-gen-fun}(3), we know that each part of $\lambda^{r/2}$ has even multiplicity. Therefore, using the map $g$ defined in \eqref{f2-map-defn}, we have that the generating function for the 
number of partitions in $\lambda^{r/2}$ 
is also 
\[
 \prod_{i=1}^{\infty}\frac{1}{1-q^{2i}} ,
\]
completing the proof. 
\end{proof}

\begin{lemma}{\cite[Observation 4.2]{adin-postnikov-roichman-2010}}
\label{lem:abs-square-roots-no} 
 \begin{enumerate}
\item Suppose $d$ is odd. Consider an element $\pi \in G(r, 1, d)$ with cycle type $\bla$ such that $\lambda^0 = (d)$. Then $\pi$ has $r$ absolute square roots.

\item Consider an element $\pi \in G(r, 1, 2d)$ with cycle type $\bla$ such that $\lambda^i = \lambda^{r-i} = (d)$ for some $i$. Then $\pi$ has $rd$ absolute square roots.
\end{enumerate}
\end{lemma}

\begin{proof}
Let $\pi=(0,\ldots,0;\sigma)$, where $\sigma = (d,d-2,\ldots,1,d-1,d-3,\ldots,4,2))\in G(r,1,d)$. Since $d$ is odd, $\sigma$ has a unique square root $(d,d-1,\ldots,1)$. Let $(z_1,\ldots,z_d;(d,d-1,\ldots,1))$ be an absolute square root of $\pi$.
By definition, we have $z_i-z_{i+1}=0$ for $1\leq i \leq d-1$ and $z_d-z_1=0$. Therefore any choice of $z_1$ is valid and fixes all other $z_i$'s. Therefore, we have $r$ possible absolute square roots of $\pi$.

Next, let $\pi=(t_1,0,\ldots,0,t_2; \sigma)$, where $\sigma=(2d,2d-2,\ldots,4,2)(2d-1,2d-3,\ldots,3,1))$ such that $t_1+t_2=0$. In $S_{2d}$, $\sigma$ has $d$ square roots. Let $(z_1,\ldots,z_{2d};(2d,2d-1,\ldots,1))$ be an absolute square root of $\pi$. Then we get the following system of linear equations:
\[
z_1-z_2=t_1, z_2-z_3=0,\ldots,z_{2d-1}-z_{2d}=0, z_{2d}-z_1=t_2.
\]
Observe that the above system has $r$ solutions because every choice of $z_2$ is valid and fixes all other values. For each square root of $\sigma$, we have $r$ absolute square roots of $\pi$. This concludes the proof.
\end{proof}

\begin{proposition}
\label{prop:num-abs-sqrt}
Given a positive integer $r$ and $t$ such that $0\leq t \leq r-1$, the number of absolute square roots of an element $\pi$ whose cycles all have the same length $\ell$ is as follows.

\begin{enumerate}
\item When $\ell = 2k$ and there are $2m$ cycles of color $0$, this is given by
$(2m-1)!! \, (2kr)^{m}$.
			
\item When $\ell = 2k+1$ and there are $m$ cycles of color $0$, this is given by
\[
\sum_{j=0}^{\lfloor \frac{m}{2}\rfloor} \binom{m}{2j} \, (2j-1)!! \, (2k+1)^j \, r^{m-j}.
\]
			
\item When $\ell = k$ and there are $m$ cycles of colors $t$ and $r-t$ each (where $a \neq r/2$), this is given by
$m! \, (k r)^{m}$.
			
\item For even $r$, when $\ell = k$ and there are $2m$ cycles of color $r/2$, this is given by
 $(2m-1)!! \, (k r)^{m}$.

\end{enumerate}
\end{proposition}

\begin{proof}
First suppose $\ell = 2k$,and there are $2m$ cycles of color $0$. Any absolute square root of $\pi$ will have $m$  cycles of length $4k$ and any cycle of length $4k$ comes from pairing two cycles of length $2k$. This pairing can be done in $(2m-1)!!$ ways. Moreover, for every pairing, the cycles can be merged in $2k$ ways and by \cref{lem:abs-square-roots-no}, the color of the $(4k)$-cycle can be any $t \in [0,r-1]$. This completes the proof of (1). The proof of (4) follows via a very similar argument.

Now suppose $\ell = 2k+1$ and there are $m$ cycles of color $0$.
Note that any absolute square root of $\pi$ has some cycles of length $2k+1$ of any color, and some cycles of length $4k+2$ of any color by \cref{lem:abs-square-color-criterion}.
We get a factor of $r$ for the color of every $(2k+1)$-cycle.
Any cycle of length $4k+2$  again comes from pairing two cycles of length $2k+1$. So, we can merge $j$ pairs by first choosing $2j$ cycles of length $2k+1$ and then we can combine them in $(2j-1)!!$ ways and then merge them in $2k+1$ ways each. Moreover, for every pairing, by \cref{lem:abs-square-roots-no}(2), the color of the $(4k+2)$-cycle can be chosen arbitrarily. This completes the proof of (2). 
	
The proof of (3) follows by the observation that we have to match one cycle of color $t$ to another cycle of color $r-t$ and then merge them. The number of matchings is clearly $m!$ and the number of ways of merging is $k^{m}$ and finally by \cref{lem:abs-square-roots-no}(2), the color of the $(2k)$-cycle can be chosen in $r$ ways. This completes the proof.  	
\end{proof}

Adin--Postnikov--Roichman~\cite[Corollary 4.3]{adin-postnikov-roichman-2010} also give a formula to count the number of absolute square roots of any element in $\Grn$. Using \cref{prop:num-abs-sqrt}, we extend their result to determine the sum of the character table in terms of generating functions.
To do so, we also need the classic generating function for the factorials as an S-fraction due to Euler~\cite{euler-1760} given by
	\begin{equation}
	\label{def-F}
	\mathcal{F}(x) = 
	S \left(x; \left( \left\lceil \frac{n+1}{2} \right\rceil \right) \right) = 
	\sum_{n \geq 0} n! \, x^n
	= \frac{1}{\ds 1 - \frac{x}{\ds 1 - \frac{x}{\ds 1 - \frac{2x}{1 - \frac{2x}{\ddots}}}}}.
	\end{equation}

Let $\bla = (\lambda^0 \mid  \dots \mid \lambda^{r-1}) \models_r n$ be an $r$-partite partition and $\Gamma_{\bla}$ be the column sum in the character table corresponding to the cycle type $\bla$. 
By \cref{thm:APR}, $\Gamma_{\bla}$ is the number of absolute square roots of an element with cycle type $\bla$.
Also, let $\mathcal{S}^{r}(x)$ be the ordinary generating function of the sequence of character table sums $(s(\Grn))_{n \geq 0}$. Note that $\mathcal{S}^{2}(x) = \mathcal{S}^{B}(x)$.

\begin{theorem}
 \label{thm:gen-fun-abs-sqrt} 
Write each partition in the $r$-partite partition $\bla = (\lambda^0 \mid \lambda^1 \mid \ldots \mid \lambda^{r-1}) \models_r n$ in frequency notation, with
$\lambda^t = \langle 1^{m_{t,1}}, 2^{m_{t,2}}, \dots \rangle$ for $0 \leq t \leq r/2$ and $\lambda^t = \lambda^{r-t}$ for $r/2 < t \leq r-1$.
Then $\Gamma_{\bla}$ is the coefficient of 
\[ 
(x_{0, 1}^{m_{0, 1}} \dots x_{0, n}^{m_{0, n}})
\dots
(x_{r-1, 1}^{m_{r-1, 1}} \dots x_{r-1, n}^{m_{r-1, n}})
\]
in
\[
\begin{cases}
\ds \prod_{k \geq 1} \left( \ds 
\mathcal{D}(2 k r x_{0,2k}^{2}) \, 
\mathcal{R}_{(2k-1)/r} (r x_{2k-1}) \,
\prod_{t = 1}^{(r-1)/2}\mathcal{F}(k r x_{t,k}^{2})  \right)
& r \text{ odd}, \\[0.5cm]
\ds \prod_{k \geq 1} \left( \ds 
\mathcal{D}(2 k r x_{0,2k}^2) \, 
\mathcal{D}(r k x_{r/2,k}^2) \,
\mathcal{R}_{(2k-1)/r} (r x_{2k-1}) \,
\prod_{t = 1}^{(r-2)/2} \mathcal{F}(k r x_{t,k}^{2}) \right)
& r \text{ even}.
\end{cases}
\]
Therefore, 
\[
\mathcal{S}^{r}(x) = 
\begin{cases}
\ds \prod_{k \geq 1} \left( \ds \mathcal{F}(k r x^{2k})^{(r-1)/2}  \,
\mathcal{D}(2 k r x^{4k}) \, \mathcal{R}_{(2k-1)/r} (r x^{2k-1}) \right)
& r \text{ odd}, \\ 
\ds \prod_{k \geq 1} \left( \ds \mathcal{F}(k r x^{2k})^{(r-2)/2} \,
\mathcal{D}(2 k r x^{4k}) \, \mathcal{D}(r k x^{2k}) \,
\mathcal{R}_{(2k-1)/r} (r x^{2k-1}) \right)
& r \text{ even}.
\end{cases}
\]
\end{theorem}

\begin{proof}
The strategy is similar to that in the proof of \cref{thm:gen-fun-sqrt}.
Suppose $\bla = (\lambda^0 \mid \dots \mid \lambda^{r-1})$ is the $r$-partite partition indexing a conjugacy class in $\Grn$ and $\pi$ be an element with cycle type $\bla$. Then \cref{prop:abs-sqr-gen-fun} gives the conditions under which $\pi$ has an absolute square root.  Thus, write each of these partitions in frequency notation as
\begin{equation}
\lambda^t = 
\begin{cases}
\langle 1^{m_{0,1}}, 2^{2m_{0,2}}, 3^{m_{0,3}}, 4^{2m_{0,4}}, \dots \rangle 
& t = 0, \\
\langle 1^{m_{t,1}}, 2^{m_{t,2}}, \dots \rangle 
& 1 \leq t < r/2 \\
\lambda^{r-t} 
& r/2 < t \leq r-1, \\
\langle 1^{2m_{r/2,1}}, 2^{2m_{r/2,2}}, \dots \rangle 
& \text{$r$ is even and $t = r/2$}.
\end{cases}
\end{equation}
As before, the number of square roots is given by \cref{prop:num-abs-sqrt} if only one of the $m_{t,j}$'s is non-zero.

We again have to argue that absolute square roots of cycles of different lengths and colors in $\lambda$ can be formed independently. For odd cycles, they just change color and for even cycles, they pair up and change color to form the absolute square root in specific ways. In either case, they do not depend on other cycles or colors.

The argument is very similar now to the proof of \cref{thm:gen-fun-sqrt}. Even (resp. odd) cycles 
of length $2k$ (resp. $2k-1$)
in $\lambda^0$ give $\mathcal{D}(2krx_{0, 2k}^{2})$ (resp. $\mathcal{R}_{(2k-1)/r} (r x_{2k-1})$) using \cref{prop:num-abs-sqrt}(1) (resp. (2)).
Cycles in $\lambda^t = \lambda^{r-t}$ of length $k$
give $\mathcal{F}(k r x_{t,k}^2)$, one for each $t$, by \cref{prop:num-abs-sqrt}(3). Lastly, when $r$ is even and the cycle is of length $2k$, we get another factor of $\mathcal{D}(k r x_{r/2, k}^2)$ by \cref{prop:num-abs-sqrt}(4).

To obtain the generating function $\mathcal{S}^r(x)$, replace each $x_{t,k}$ by $x^k$.	
\end{proof}

\begin{remark}
\label{rem:Bn}
When $r = 2$, absolute square roots are exactly the usual square roots. Thus the generating function for $(s_n^B)$ can be obtained by setting $r=2$ in either \cref{thm:gen-fun-sqrt} or \cref{thm:gen-fun-abs-sqrt}.
\end{remark}

\section*{Acknowledgements}
We thank Jyotirmoy Ganguly for helpful discussions.
This work benefited immensely from Sagemath~\cite{sagemath} and GAP~\cite{GAP4}.
We acknowledge support from the DST FIST program - 2021 [TPN - 700661].
The first author (AA) was partially supported by SERB Core grant CRG/2021/001592.

\bibliography{kron}
\bibliographystyle{alpha}

\appendix

\section{Exceptional irreducible Coxeter groups}
\label{sec:exceptional-groups}

For each of the exceptional simple Lie groups, we list the total sum and the first column sum of their Weyl groups below.

\begin{table}[h!]
\begin{tabular}{|c|c|c|}
\hline
$G$ & $\Gamma_e(G)$ & $s(G)$ \\
\hline
$E_6$ & 892 & 995 \\
$E_7$ & 10,208 & 10,734 \\
$E_8$ & 199,952 & 220,772 \\
$F_4$ & 140 & 200 \\
$G_2$ & 8 & 10 \\
$H_2$ & 32 &40\\
$H_3$ & 572 &770\\
\hline
\end{tabular}
\end{table}

Note that the Weyl group of $G_2$ is the dihedral group of order 12. In each case, it is easy to check that the group satisfied \cref{property:S}. 

\section{Small demioctahedral groups}
\label{sec:small demi}

For small demioctahedral groups, we have the following table:

\begin{table}[h!]
\begin{tabular}{|c|c|c|}
\hline
$G$ & $\Gamma_e(G)$ & $s(G)$ \\
\hline
$D_6$ & 752 & 930 \\
$D_7$ & 3,256 & 4,037 \\
$D_8$ & 17,040 & 21,796 \\
$D_9$ & 84,496 & 99,525\\
$D_{10}$ & 475,712 & 542,616 \\
$D_{11}$ & 2,611,104 & 2,961,697 \\
$D_{12}$ & 15,687,872 & 18,040,858 \\
$D_{13}$ & 93,376,960 & 103,201,617 \\
$D_{14}$ & 594,638,592 & 647,826,742 \\
$D_{15}$ & 3,786,534,784 & 4,109,646,977 \\
\hline
\end{tabular}
\end{table}

\section{Data for small groups}
\label{sec:smallgroups}

Using the \texttt{SmallGroups} library in \texttt{GAP} within \texttt{SageMath}, we looked at all finite groups of order up to 200. We found groups where \cref{property:S} fails (listed in \cref{sec:prop-fails}) and groups where the first column sum equals the sum of the other columns (listed in \cref{sec:prop-equal}).
We also found that the list of ratios $s(G)/\Gamma_e(G)$ among all these groups is quite small. We have listed these ratios in \cref{sec:ratios}.

\subsection{Groups $G$ where $s(G) > 2\Gamma_e(G)$}
\label{sec:prop-fails}

There are very few groups where the sum of the entries of the character table of a finite group is more than twice the sum of dimensions of its irreducible representations and we list them below. The only orders where equality holds up to order $200$ are $64, 125, 162$ and $192$ and there are $5, 2, 542, 2$ and $210$ nonisomorphic groups of these orders respectively. We list the index in the \texttt{SmallGroups} library for each of these orders below.

\begin{enumerate}
\item [64] : $241, 242,  243, 244, 245.$

\item [125] : $3,  4.$

\item [128] : 
$36, 37, 38, 39, 40, 41, 71, 72, 73, 74, 87, 88, 138, 139, 144, 145, 242, 243, 244, 245, 246, 247, 265, 266, 267,$

\noindent
$268, 269,287, 288, 289, 290, 291, 292, 293, 387, 388, 389, 390, 391, 392, 393, 394, 395, 396, 417, 418, 419,$

\noindent
$420, 421, 422, 423, 424,  425, 426, 427, 428, 429, 430, 431, 432, 433, 434, 435, 436, 527, 528, 560, 561, 562,$

\noindent
$560, 561, 562, 634, 635, 636, 637, 641, 642, 643, 644, 645, 646, 647, 740, 741, 764, 780, 781, 801, 802, 813,$

\noindent
$814, 836, 853, 854, 855, 859, 860, 861, 865, 866, 867, 912, 913, 922, 923, 931, 932, 933, 934, 935, 970, 971,$

\noindent
$1107, 1108, 1109, 1110, 1111, 1112, 1113, 1114, 1115, 1345, 1346, 1347, 1348, 1349, 1350, 1351, 1352, 1353,$

\noindent
$1354, 1355, 1356, 1357, 1358, 1359, 1360, 1361, 1362, 1363, 1364, 1365, 1366, 1367, 1368, 1369, 1370, 1371,$

\noindent 
$1372, 1373, 1374, 1375, 1376, 1377, 1378, 1379, 1380, 1381, 1382, 1383, 1384, 1385, 1386, 1387, 1388, 1389,$

\noindent 
$1390, 1391, 1392, 1393,1394, 1395, 1396, 1397, 1398, 1399, 1400, 1401, 1402, 1403, 1404, 1405, 1406, 1407,$

\noindent
$1408, 1409, 1410, 1411, 1412, 1413, 1414, 1415, 1416, 1417, 1418, 1419, 1420, 1421, 1422, 1423, 1424, 1425,$

\noindent
$1426, 1427,1428, 1429, 1430, 1431, 1432, 1433, 1434, 1435, 1436, 1437, 1438, 1439, 1440, 1441, 1442, 1443,$

\noindent 
$1444,1445, 1446, 1447, 1448, 1449, 1450, 1451, 1452, 1453, 1454, 1455, 1456, 1457, 1458, 1459, 1460, 1461,$

\noindent
$1462, 1463, 1464, 1465, 1466, 1467, 1468, 1469, 1470, 1471, 1472, 1473, 1474, 1475, 1476, 1477, 1478, 1479,$

\noindent
$1480, 1481, 1482, 1483, 1484, 1485, 1486, 1487, 1488, 1489, 1490, 1491, 1492, 1493, 1494, 1495, 1496, 1497,$

\noindent
$1498, 1499, 1500, 1501, 1502, 1503, 1504, 1505, 1506, 1507, 1508, 1509, 1510, 1511, 1512, 1513, 1514, 1515,$

\noindent 
$1516, 1517, 1518, 1519, 1520, 1521, 1522, 1523, 1524, 1525, 1526, 1527, 1528, 1529, 1530, 1531, 1532, 1533,$

\noindent
$1534, 1535, 1536, 1537, 1538, 1539, 1540, 1541, 1542, 1543, 1544, 1545, 1546, 1547, 1548, 1549,  1550, 1551,$

\noindent
$ 1552, 1553, 1554, 1555, 1556, 1557, 1558, 1559, 1560, 1561, 1562, 1563, 1564, 1565, 1566, 1567, 1568, 1569,$

\noindent
$1570, 1571, 1572, 1573, 1574, 1575, 1576, 1577, 1710, 1712, 1719, 1727, 1751, 1752, 1753, 1754, 1758, 1759,$

\noindent
$1760, 1800, 1801, 1924, 1925, 1926, 1927, 1928, 1929, 1930, 1931, 1932, 1933, 1934, 1935, 1936, 1945,  1946,$

\noindent
$1947, 1948, 1949, 1950, 1951, 1952, 1953, 1954, 1955, 1956, 1957, 1966, 1967, 1968, 1969, 1970, 1971, 1972,$

\noindent
$1973, 1974, 1975, 1976, 1983, 1984, 1985, 1986, 1987, 1988, 1989, 1990, 1991, 1992, 1993, 1994, 1995, 1996,$

\noindent
$1997, 1998, 1999, 2000, 2001, 2002, 2003, 2004, 2005, 2006, 2007, 2008, 2009, 2010, 2020, 2021,  2038, 2039,$

\noindent
$2040, 2041, 2042, 2043, 2044, 2045, 2046, 2047, 2048, 2049, 2050, 2051, 2052, 2053, 2054, 2055, 2056, 2057,$

\noindent
$2058, 2059, 2060, 2061, 2062, 2063, 2064, 2065, 2075, 2076, 2077, 2078, 2079, 2080, 2081, 2082, 2083, 2084,$

\noindent
$2085, 2086, 2087, 2088, 2089, 2098, 2099, 2100, 2101, 2102, 2103, 2104, 2105, 2106, 2107, 2108, 2109,  2116,$

\noindent
$2117, 2118, 2119, 2120, 2121, 2122, 2129, 2130, 2131, 2132, 2133, 2134, 2135, 2257, 2258, 2259,  2260, 2261,$

\noindent
$2262, 2263, 2264, 2265, 2266, 2267, 2268, 2269, 2270, 2271, 2272, 2273, 2274, 2275, 2276, 2277, 2278, 2279,$

\noindent
$2280, 2281, 2282, 2283, 2284, 2285, 2286, 2287, 2288, 2289, 2290, 2291, 2298, 2299, 2300.$

\item[162] : $35,  162.$

\item [192] : $30, 31, 32, 33, 34, 35, 36, 37, 299, 300, 301, 302, 303, 304, 305, 306, 307, 308, 309, 310, 311, 312,  313, 314,$

\noindent
$323, 324, 339, 341, 353, 356, 371, 375, 415, 424, 426, 435, 445, 470, 473, 474, 477, 598, 601, 608, 611, 614,$

\noindent
$617, 619, 623, 626, 628, 630, 633, 640, 641, 645, 648, 713, 718, 726, 733, 735, 748, 749, 757, 758, 759, 760,$
 
\noindent
$ 761, 762, 763, 764, 1146, 1148, 1149, 1150, 1151, 1152, 1153, 1154, 1156, 1157, 1158, 1160, 1161, 1162, 1163,$
 
\noindent
$1164, 1166, 1167, 1168, 1169, 1170, 1171, 1172, 1173, 1174, 1175, 1176, 1177, 1178, 1179, 1180, 1182, 1184,$
 
\noindent
$ 1187, 1188, 1189, 1190, 1191, 1192, 1193, 1194, 1195, 1196, 1197, 1198, 1199, 1200, 1201, 1202, 1203, 1204,$
  
\noindent
$1205, 1206, 1207, 1209, 1210, 1212, 1213, 1214, 1216, 1217, 1218, 1219, 1220, 1221, 1222, 1223, 1224, 1225,$
   
\noindent 
$1226, 1228, 1229, 1230, 1231, 1233, 1234, 1235, 1236, 1237, 1238, 1240, 1241, 1242, 1243, 1244, 1245, 1248,$
   
\noindent
$1249, 1251, 1252, 1253, 1254, 1255, 1256, 1257, 1258, 1259, 1260, 1261, 1263, 1264, 1266, 1267, 1268, 1269,$
    
\noindent
$1270, 1271, 1272, 1274, 1276, 1277, 1278, 1279, 1280, 1281, 1283, 1285, 1286, 1288, 1289, 1290, 1291, 1292,$
     
\noindent
$1293, 1294, 1331, 1332, 1333, 1334, 1335, 1336, 1337, 1338, 1449, 1450, 1451, 1452, 1453.$

\end{enumerate}

\subsection{Groups $G$ where $s(G) = 2\Gamma_e(G)$}
\label{sec:prop-equal}

There are very few groups where the sum of the entries of the character table of a finite group is equal to twice the sum of dimensions of its irreducible representations and we list them below. The only orders where equality holds up to order 200 are $64, 128, 160$ and $192$ and there are $23, 642, 4$ and $135$ nonisomorphic groups of these orders respectively. We list the index in the \texttt{SmallGroups} library for each of these orders below.

\begin{enumerate}

\item[64]: $18, 19, 28, 149, 150, 151, 170, 171, 172, 177, 178, 182, 215, 216, 217, 218, 219, 220, 221, 222, 223, 224, 225$.

\item[128]: $6, 7, 16, 17, 18, 19, 20, 21, 22, 23, 24, 25, 45, 107, 130, 194, 195, 196, 197, 198, 199, 200, 201, 202, 203, 204$, 

\noindent
$205, 227, 228, 229, 301, 324, 325, 326, 327, 328, 329, 330, 331, 332, 333, 334, 335, 336, 337, 338, 339, 340$, 

\noindent
$341, 342, 343, 344, 345, 346, 347, 348, 349, 350, 446, 447, 448, 449, 450, 451, 452, 453, 454, 455, 462, 463$, 

\noindent
$479, 513, 514, 515, 516, 517, 531, 532, 533, 541, 543, 544, 545, 551, 552, 553, 554, 559, 565, 568, 570, 573$, 

\noindent
$579, 581, 626, 627, 629, 630, 631, 632, 633, 667, 668, 670, 675, 676, 678, 691, 692, 693, 695, 696, 697, 698$, 

\noindent
$699, 703, 704, 705, 707, 724, 725, 727, 728, 729, 730, 734, 735, 736, 737, 738, 739, 742, 749, 750, 751, 752$,

\noindent
$753, 754, 759, 760, 761, 762, 763, 769, 770, 771, 772, 776, 777, 778, 779, 782, 783, 784, 785, 792, 793, 794$, 

\noindent
$795, 796, 810, 811, 812, 821, 822, 823, 830, 834, 835, 841, 842, 897, 898, 950, 951, 952, 975, 976, 977, 982$, 

\noindent
$983, 987, 1040, 1041, 1042, 1043, 1044, 1045, 1046, 1047, 1048, 1049, 1050, 1051, 1052, 1053, 1054, 1055$,

\noindent
$1056, 1057, 1058, 1059, 1060, 1061, 1062, 1063, 1064, 1065, 1066, 1067, 1068, 1069, 1135, 1136, 1137, 1138$, 

\noindent
$1139, 1140, 1141, 1142, 1143, 1144, 1145, 1146, 1147, 1148, 1149, 1150, 1151, 1152, 1153, 1154, 1155, 1156$, 

\noindent
$1157, 1158, 1159, 1160, 1161, 1162, 1163, 1164, 1165, 1166, 1167, 1168, 1169, 1170, 1171, 1172, 1173, 1174$, 

\noindent
$1175, 1176, 1177, 1178, 1179, 1180, 1181, 1182, 1183, 1184, 1185, 1186, 1187, 1188, 1189, 1190, 1191, 1192$, 

\noindent
$1193, 1194, 1195, 1196, 1197, 1198, 1199, 1200, 1201, 1202, 1203, 1204, 1205, 1206, 1207, 1208, 1209, 1210$, 

\noindent
$1211, 1212, 1213, 1214, 1215, 1216, 1217, 1218, 1219, 1220, 1221, 1222, 1223, 1224, 1225, 1226, 1227, 1228$,

\noindent
$1229, 1230, 1231, 1232, 1233, 1234, 1235, 1236, 1237, 1238, 1239, 1240, 1241, 1242, 1243, 1244, 1245, 1246$, 

\noindent
$1247, 1248, 1249, 1250, 1251, 1252, 1253, 1254, 1255, 1256, 1257, 1258, 1259, 1260, 1261, 1262, 1263, 1264$, 

\noindent
$1265, 1266, 1267, 1268, 1269, 1270, 1271, 1272, 1273, 1274, 1275, 1276, 1277, 1278, 1279, 1280, 1281, 1282$, 

\noindent
$1283, 1284, 1285, 1286, 1287, 1288, 1289, 1290, 1291, 1292, 1293, 1294, 1295, 1296, 1297, 1298, 1299, 1300$, 

\noindent
$1301, 1302, 1303, 1304, 1305, 1306, 1307, 1308, 1309, 1310, 1311, 1312, 1313, 1314, 1315, 1316, 1317, 1318$, 

\noindent
$1319, 1320, 1321, 1322, 1323, 1324, 1325, 1326, 1327, 1328, 1329, 1330, 1331, 1332, 1333, 1334, 1335, 1336$, 

\noindent
$1337, 1338, 1339, 1340, 1341, 1342, 1343, 1344, 1611, 1615, 1616, 1620, 1621, 1637, 1653, 1655, 1662, 1664$, 

\noindent
$1693, 1699, 1705, 1707, 1708, 1713, 1715, 1717, 1721, 1724, 1725, 1735, 1736, 1737, 1738, 1739, 1740, 1741$, 

\noindent
$1742, 1743, 1744, 1745, 1768, 1769, 1770, 1771, 1772, 1773, 1774, 1775, 1776, 1777, 1778, 1783, 1784, 1785$,

\noindent
$1786, 1787, 1788, 1789, 1790, 1791, 1792, 1793, 1794, 1795, 1809, 1810, 1811, 1812, 1813, 1814, 1815, 1816$,

\noindent
$1824, 1825, 1826, 1827, 1828, 1829, 1830, 1831, 1841, 1842, 1843, 1844, 1845, 1846, 1847, 1848, 1849, 1850$,

\noindent
$1851, 1852, 1853, 1854, 1855, 1856, 1857, 1858, 1859, 1864, 1865, 1866, 1867, 1868, 1869, 1870, 1871, 1872$,

\noindent
$1873, 1874, 1880, 1881, 1882, 1883, 1884, 1885, 1886, 1887, 1888, 1893, 1894, 1895, 1896, 1897, 1898, 1903$,

\noindent
$1904, 1905, 1906, 1907, 1908, 1909, 1910, 1911, 1912, 1913, 1914, 1915, 1916, 1917, 1918, 1919, 1920, 1921$,

\noindent
$1922, 1923, 1937, 1938, 1939, 1940, 1941, 1942, 1943, 1944, 1958, 1959, 1960, 1961, 1962, 1963, 1964, 1965$,

\noindent
$1977, 1978, 1979, 1980, 1981, 1982, 2177, 2178, 2179, 2180, 2181, 2182, 2183, 2184, 2185, 2186, 2187, 2188$,

\noindent
$2189, 2190, 2191, 2192, 2193, 2216, 2217, 2218, 2219, 2220, 2221, 2222, 2223, 2224, 2225, 2226, 2227, 2228$,

\noindent
$2229, 2230, 2231, 2232, 2233, 2234, 2235, 2236, 2237, 2238, 2239, 2240, 2241, 2242, 2243, 2244, 2245, 2246$,

\noindent
$2247, 2248, 2249, 2250, 2251, 2252, 2253, 2254, 2255, 2256, 2317, 2318.$

\item[160]: $132, 135, 136, 139.$

\item[192]: $25, 71, 90, 116, 119, 143, 144, 153, 261, 262, 271, 272, 273, 274, 381, 384, 385, 386, 455, 456, 457, 592, 593$, 

\noindent
$594, 595, 596, 599, 602, 603, 604, 605, 606, 609, 613, 616, 622, 625, 631, 634, 642, 643, 646, 649, 650, 653$, 

\noindent
$693, 694, 696, 719, 736, 751, 800, 801, 802, 803, 804, 805, 901, 902, 903, 922, 923, 924, 929, 930, 934, 1042$, 

\noindent
$1052, 1053, 1059, 1069, 1074, 1077, 1078, 1084, 1085, 1093, 1094, 1119, 1120, 1121, 1123, 1141, 1142, 1145$,

\noindent
$1147, 1155, 1159, 1163, 1165, 1181, 1183, 1185, 1186, 1208, 1211, 1215, 1227, 1232, 1239, 1246, 1247, 1250$,

\noindent
$1262, 1265, 1273, 1275, 1282, 1284, 1287, 1363, 1364, 1375, 1384, 1389, 1392, 1394, 1395, 1396, 1397, 1423$,

\noindent
$1424, 1425, 1426, 1427, 1428, 1429, 1430, 1431, 1432, 1433, 1524, 1525, 1526, 1527.$

\end{enumerate}

\subsection{Set of distinct ratios $s(G)/\Gamma_e(G)$ up to order $200$}
\label{sec:ratios}

Among all nonisomorphic groups of order up to $200$, of which there are $6065$, we find only $176$ distinct ratios of $s(G)/\Gamma_e(G)$. These are
listed in increasing order below:

\noindent
$1,
 31/28,
 7/6,
 19/16,
 31/26,
 11/9,
 5/4,
 19/15,
 51/40,
 13/10,
 37/28,
 4/3,
 11/8,
 25/18,
 7/5,
 17/12,
 10/7,
 23/16,$
 
 \noindent
 $
 13/9,
 29/20,
 16/11,
 35/24,
 19/13,
 41/28,
 22/15,
 47/32,
 25/17,
 53/36,
 28/19,
 59/40,
 31/21,
 65/44,
 34/23,
 71/48,$
 
 \noindent
 $
 37/25,
 77/52,
 40/27,
 83/56,
 43/29,
 89/60,
 46/31,
 95/64,
 49/33,
 101/68,
 52/35,
 107/72,
 55/37,
 113/76,
 58/39,$
 
\noindent
 $
 119/80
 61/41,
 125/84,
 64/43,
 131/88,
 67/45,
 137/92,
 70/47,
 143/96,
 73/49,
 149/100,
 76/51,
 3/2,
 41/27,
 38/25,$
 
 \noindent
 $
 32/21,
 26/17,
 23/15,
 20/13,
 37/24,
 17/11,
 31/20,
 14/9,
 25/16,
 47/30,
 11/7,
 19/12,
 73/46,
 35/22,
 43/27,
 51/32,$
 
\noindent
$8/5,
 53/33,
 71/44,
 21/13,
 34/21,
 73/45,
 13/8,
 83/51,
 44/27,
 103/63,
 18/11,
 23/14,
 43/26,
 53/32,
 73/44,
 5/3,$
 
 \noindent
 $
 67/40,
 57/34,
 47/28,
 37/22,
 59/35,
 27/16,
 22/13,
 61/36,
 39/23,
 17/10,
 29/17,
 41/24,
 12/7,
 79/46,
 55/32,
 19/11,$
 
 \noindent
 $
 26/15,
 33/19,
 40/23,
 7/4,
 51/29,
 23/13,
 85/48,
 62/35,
 55/31,
 16/9,
 25/14,
 34/19,
 43/24,
 115/64,
 9/5,
 101/56,$
 
 \noindent
 $
 65/36,
 47/26,
 29/16,
 20/11,
 53/29,
 64/35,
 11/6,
 35/19,
 24/13,
 13/7,
 28/15,
 47/25,
 32/17,
 17/9,
 121/64,
 36/19,$
 
 \noindent
 $
 19/10,
 99/52,
 40/21,
 23/12,
 25/13,
 52/27,
 31/16,
 33/17,
 35/18,
 37/19,
 2,
 65/32,
 35/17,
 29/14,
 25/12,
 23/11,$
 
 \noindent
 $
 21/10,
 40/19,
 36/17,
 17/8,
 28/13,
 13/6,
 35/16,
 11/5,
 20/9,
 9/4,
 25/11,
 16/7,
 7/3,
 40/17,
 8/3,
 25/9.$

\end{document}